\documentclass[onecolumn]{autart}    % Enable this line and disable the 
\usepackage{graphicx}     
\usepackage{epstopdf}
\usepackage{amsmath,amssymb,amsfonts}
\usepackage{color}
\usepackage{lipsum}
\usepackage{epstopdf}
\usepackage{amsopn}

\usepackage{natbib}

\usepackage{tikz}
\usetikzlibrary{arrows}
\usepackage{verbatim}
\usepackage{algorithm}
\usepackage{algpseudocode}
\usepackage{calc}
\usepackage{mathrsfs}  
\usepackage[normalem]{ulem}
\usepackage[english]{babel}

\newcommand{\bmat}{\left[ \begin{matrix}}
	\newcommand{\emat}{\end{matrix} \right]}
\newcommand{\innerprod}[2]{\langle{#1},\,{#2}\rangle}

\DeclareMathOperator*{\argmin}{argmin} %GF declared new operator

\newcommand{\Rbb}{\mathbb R}

\newcommand{\Nbb}{\mathbb N}

\usepackage{dsfont}

\newcommand{\pb}{\mathbf  p}

\newcommand{\Ccal}{\mathcal{C}}
\newcommand{\Dcal}{\mathcal{D}}

\newcommand{\Bcal}{\mathcal{B}}
\newcommand{\Rcal}{\mathcal{R}}
\newcommand{\Scal}{\mathcal{S}}

\newcommand{\Hcal}{\mathcal{H}}
\newcommand{\Kcal}{\mathcal{K}}

\newcommand{\Xcal}{\mathcal{X}}

\newtheorem{theorem}{Theorem}
\newtheorem{remark}[theorem]{Remark}

\newtheorem{lemma}[theorem]{Lemma}
\newtheorem{definition}[theorem]{Definition}
\newtheorem{proposition}[theorem]{Proposition}
\newtheorem{corollary}[theorem]{Corollary}

\newtheorem{assumption}{Assumption}

\renewcommand{\alg}[1]{\begin{align} #1 \end{align}}
\newcommand{\nn}{\nonumber}

\newcommand{\fix}[1]{{\color{black} #1}}

\newcommand\norm[1]{\left\lVert#1\right\rVert} % GF norm of 

\begin{document}
\begin{frontmatter}

\title{Identification of forward models:\\ a nonparametric approach} % Title, preferably not more 

%\thanks[footnoteinfo]{TBA}

\author[unipd]{Giulio Fattore}\ead{fattoregiu@dei.unipd.it},    % Add the 
\author[unipd]{Marco Peruzzo}\ead{peruzzomar@dei.unipd.it},               % e-mail address 
\author[Svezia]{Giacomo Sartori}\ead{giacomo.sartori@ntnu.no},  % (ead) as shown
\author[unipd]{Mattia Zorzi}\ead{zorzimat@dei.unipd.it}  % (ead) as shown

\address[unipd]{Department of Information Engineering, University of Padova, Via Gradenigo 6/B, 35131 Padova, Italy}  % Please supply                                              
\address[Svezia]{Department of Chemical Engineering, NTNU
N- 7491 Trondheim, Norway}             % full addresses
       % here.

                             % keyword list or with the 
                                          % help of the Automatica 
                                          % keyword wizard

\begin{abstract}                     
 In this paper we propose a new kernel-based method for the identification of the impulse responses of forward models. The resulting estimator leads to a nonlinear Tikhonov regularization problem for which we prove the existence of a solution. The latter result makes legitimate to approximate the forward model through a high-order Moving Average with eXogenous  input (MAX)
  model, i.e. the numerical solution found using this model introduces only a negligible bias in the estimate. The optimization of the marginal likelihood to estimate the kernel hyperparameters is also taken into account. Since there does not exist a closed-form expression for the marginal likelihood, we present an evaluation method that relies on the Laplace approximation of the marginal likelihood. Finally, some numerical experiments are discussed to show the effectiveness of the proposed method.
\end{abstract}

\end{frontmatter}

\section{Introduction}

A common problem that often appears in various fields of science and engineering is the identification of a dynamical model from input-output observations. In the context of linear discrete-time systems, a well established paradigm is the Prediction Error Method (PEM), see \cite{LJUNG_SYS_ID_1999,SODERSTROM_STOICA_1988}. The challenging aspect of such approach is the selection of the model class which is typically addressed using the Akaike Information Criterion (AIC) or Bayesian Information Criterion (BIC), see \cite{AKAIKE_1974,SCHWARZ_1978}.
A way to overcome this issue is to infer the impulse responses characterizing the system from the observations. The corresponding problem, which is infinite-dimensional,  is inherently ill-posed, as it relies on a finite set of available observations, and regularization is needed. The resulting approach is known as kernel-based method \citep{PILLONETTO_DENICOLAO2010,EST_TF_REVISITED_2012,doi:10.1080/00207179.2019.1578407}. From a Bayesian perspective, the impulse responses are modeled as  zero-mean Gaussian processes with a certain covariance function (or simply kernel), see \cite{PILLONETTO_2011_PREDICTION_ERROR}. Therefore, the prior information on the impulse responses is introduced by assigning an appropriate kernel. Examples of prior information on the impulse responses are the requirement to be absolutely summable, i.e. the corresponding system is Bounded Input Bounded Output (BIBO) stable, and possesses a certain level of smoothness. Many approaches have been proposed to design kernels for system identification, see e.g.  \cite{CHEN2018109,ZORZI2018125,marconato2017filter,zorzi2021second,chen2015kernel1,chen2015kernel2,FUJI1,FUJI2,chen2014system,10039069,BKRON}.

In this paper our attention is devoted to learning the impulse responses of a forward (or simulation) model, namely the output is the sum of the filtered input and filtered white noise. The common way to tackle  this problem is to re-parametrize the model using the impulse responses of the one-step ahead predictor, i.e. the output is the sum of the one-step ahead output predictor and white noise, and learn the latter by means of the kernel-based method \citep{PILLONETTO_2011_PREDICTION_ERROR}. However, it is very difficult, if not almost impossible, to induce some a priori information on the impulse responses of the forward model using this approach. For instance, in order to obtain a BIBO stable forward model it is necessary to modify the kernel-based method using a sequential strategy \citep{PILLONETTO2022110169}. In general, the a priori information on the impulse responses of the forward model is more complex:
for instance, to model the generation of synchronized burst or coherence resonance phenomena in neural systems it is necessary to impose certain time correlation conditions on one impulse response of the forward model \citep{guo2011inhibition,SilvaViella}.

In this paper we propose a new kernel-based method to estimate forward models.  This approach allows us to frame the problem directly in terms of the impulse responses of the forward model, making it possible to embed the a priori information on the forward parameterization. It turns out this paradigm is characterized by a nonlinear infinite-dimensional Tikhonov regularization problem. We prove such problem does  admit solution under a reasonable hypothesis on the initial conditions. Moreover, we characterize the structure of all the possible solutions and the latter represents a generalization of the representer theorem  \citep{kimeldorf1970correspondence}. Finally, this existence result allows to approximate the forward model through 
a high-order Moving Average with eXogenous  input (MAX) model, i.e. the numerical solution found using a high-order MAX model introduces only a negligible bias in the estimate.

The kernels depend on some hyperparameters which are usually inferred from the data by optimizing the so-called marginal likelihood \citep{RASMUSSEN_WILLIAMNS_2006,chen2013implementation}.

Another significant challenge in our setup is the difficulty of deriving a closed-form expression for the marginal  likelihood: its evaluation involves the computation of an integral. Therefore, we propose a method to evaluate the marginal likelihood, that makes use of the Laplace approximation, and illustrate its effectiveness through numerical simulations. This paper builds upon some preliminary results without proof presented in the conference paper  \citep{BJREG_CONF}.

The paper is organized as follows. 
In Section~\ref{sec:prob_form}, we introduce the problem setup, namely the nonlinear Tikhonov regularization problem and the problem to learn the hyperaparameters. In Section~\ref{sec:exist_sol}, we establish the existence of a solution to the Tikhonov problem and the corresponding representation. The  procedure to evaluate  the approximated marginal likelihood for estimating the hyperparameters is discussed in Section~\ref{hyperparameter estimation}. In Section~\ref{sec:numres} numerical experiments showing the effectiveness of the proposed method are presented. Finally in Section~\ref{sec:conclusion} we draw the conclusions.

{\em Notation.} $\Nbb$ denotes the set of   \fix{natural numbers}, $\Rbb$ denotes the set of real numbers and $\mathbb{S}$ denotes the space of all sequences. Given an infinite dimensional matrix $\Kcal\in\Rbb^{\infty\times \infty}$, $\Kcal^\top$ denotes the transpose of $\Kcal$, $[\Kcal]_{ij}$ denotes its entry in position $(i,j)$. The definitions are similar in the case in which we consider finite dimensional matrices. We denote with $\ell^2$ the  Hilbert space of absolutely square summable infinite length sequences, namely if $\vartheta\in\ell^2$, \fix{$\norm{\vartheta}^2_2=\sum_{k=1}^\infty\vartheta_k^2<\infty$}, and with $\ell^1$ the vector space of absolutely summable infinite length sequences  namely if $\vartheta\in\ell^1$, \fix{$\norm{\vartheta}_1=\sum_{k=1}^\infty|\vartheta_k|<\infty$}. The notation $y\sim \mathcal N(\mu,K)$ means that $y$ is a Gaussian random vector with mean $\mu$ and covariance matrix $K$. 

\section{Problem Formulation} \label{sec:prob_form}
Consider the nonparametric Single Input Single Output (SISO) linear forward model 
\begin{align}\label{def_mod_inf}
    y(t)=\sum_{k=1}^{\infty} b_k u(t-k)+\sum_{k=0}^\infty c_k e(t-k),
\end{align}
where $c_0=1$; $y(t),u(t)$ denote the output and the input of the system at time $t$, respectively; $e(t)$ is White Gaussian Noise (WGN) with variance $\sigma^2$. The input can be deterministic or a stochastic process independent of $e$. We assume that the system is BIBO stable that is the corresponding impulse responses are absolutely summable, i.e.
\begin{align*}
   b =\begin{bmatrix}
       b_1\\
        b_2\\
        \vdots
    \end{bmatrix}\in \ell^1, \quad     c =\begin{bmatrix}
        c_1\\
        c_2\\
        \vdots
    \end{bmatrix}\in \ell^1.
\end{align*}
Assume to collect a dataset $\Dcal:=\{(y(t),u(t)),\; t=1,\, \ldots,\, N\}$ generated by 
(\ref{def_mod_inf}). We want to estimate the impulse responses characterizing model (\ref{def_mod_inf})  that is $\vartheta=(b ,c )$ from $\Dcal$. We denote with
\begin{align*}
\mathrm u^-&:=[\, u(0)\; u(-1) \,\ldots \,]^\top,\\
\mathrm y^-&:=[\, y(0)\; y(-1) \,\ldots \,]^\top,
\end{align*}
the (infinite length) past data vector of the input and output, respectively. We also define the vectors of the collected data as 
\begin{align*}
\mathrm u^+&:=[\, u(1)\; u(2) \,\ldots\, u(N)\,]^\top,\\
\mathrm y^+&:=[\, y(1)\; y(2) \,\ldots\, y(N)\,]^\top.\end{align*}
Let
\begin{align*}
    B(z)=\sum_{k=1}^{\infty}b_kz^{-k},\quad C(z)=1+\sum_{k=1}^{\infty}c_kz^{-k}
\end{align*}
denote the Z-transform of the impulses $b $ and $c $. It is not difficult to see that the one-step ahead predictor of $y(t)$ under model (\ref{def_mod_inf}) is 
\begin{align} \label{hatytheta}
    \hat{y}_\vartheta(t|t-1)=[
    1-C(z)^{-1}]y(t)+\frac{B(z)}{C(z)}u(t)
\end{align}
and the corresponding prediction error is  
\begin{align}\label{eq:pred_err}
    \varepsilon_\vartheta(t)=\frac{1}{C(z)}\left[y(t)-B(z)u(t)\right],
\end{align}
where we made explicit its dependence on $\vartheta$. One would find $\vartheta$ by minimizing the average squared prediction error
\begin{align}\label{eq:Vn}
    V_N(\vartheta)=\frac{1}{N}\sum_{t=1}^N\varepsilon_{\vartheta}^2(t).
\end{align}
However such a problem is ill-posed because we have a finite number of measurements while $\vartheta$ contains infinite parameters.
To overcome this issue in the estimation of $\vartheta$, we consider the kernel-based approach proposed in \cite{PILLONETTO_DENICOLAO2010}, i.e., $\vartheta$ is obtained solving the Tikhonov regularization problem 
\begin{align}\label{eq:tikhonv_problem}
    \min_{\vartheta \in \Hcal_b\times\Hcal_c}{V_N(\vartheta)+\gamma_b\norm{b }^2_{\Hcal_b}+\gamma_c\norm{c }^2_{\Hcal_c}},
\end{align}
where $\Hcal_b,\Hcal_c\subseteq \ell^2$ are the Reproducing Kernel Hilbert Spaces (RKHS) with inner products $\innerprod{\cdot}{\cdot}_{\Hcal_b}$, $\innerprod{\cdot}{\cdot}_{\Hcal_c}$ and norms $\norm{\cdot}_{\Hcal_b}$, $\norm{\cdot}_{\Hcal_c}$, respectively. The corresponding kernels are denoted by $\Kcal_b$, $\Kcal_c$, while $\gamma_b=\frac{\sigma^2}{N}\frac{1}{\lambda_b}$, $\gamma_c=\frac{\sigma^2}{N}\frac{1}{\lambda_c}$ are the scale factors corresponding to the hyperparameters $\lambda_b,\lambda_c>0$.
The most popular kernel used in system identification is the 
     Tuned Correlated (TC) kernel \citep{EST_TF_REVISITED_2012}
\begin{equation}
 \label{defTC}   [\Kcal_{TC}]_{ij}=\beta^{max\{i,j\}},
\end{equation}
where $\beta\in(0,1)$ is the hyperparameter describing the decay rate of the impulse response. A more general kernel is the Second order Diagonal correlated (D2) kernel \citep{zorzi2021second}
\begin{align}
    \label{dDC2}
     [\Kcal_{D2}]_{ij}=\frac{1}{1-\alpha}&\left[\beta^{max\{i,j\}}\left(1-\left(1-\beta\right)\alpha^{\left(\left|i-j\right|+1\right)}\right)\right.\nn\\
    &\ \left.-\alpha^2\beta^{max\{i,j\}+1}\right], 
    \end{align}
where the hyperparameter $\alpha\in [0,1)$. These kernels encode the a priori information that the impulse response decays to zero and has a certain degree of smoothness. More precisely, the D2 kernel induces more smoothness on the impulse response than the TC kernel: the smoothness degree is tuned according to $\alpha$. If more smoothness is needed, then it is possible to use the Stable-Spline (SS) kernel \citep{PILLONETTO_DENICOLAO2010} or the second order TC (TC2) kernel \citep{zorzi2021second}.

\begin{remark} The typical way to estimate model (\ref{def_mod_inf}) is to re-parametrize it as an infinite order ARX model 
\alg{\label{par_inf_ARX}y(t)=\sum_{k=1}^\infty a_k y(t-k)+\sum_{k=1}^{\infty} f_k u(t-k)+ e(t)} 
and then consider the corresponding Tikhonov regularization problem with impulse responses $a=[\,a_1 \; a_2 \; \ldots \,]^\top\in\ell^1$, $f=[\,f_1 \; f_2 \; \ldots \,]^\top\in\ell^1$ of the one-step ahead predictor \citep{PILLONETTO_2011_PREDICTION_ERROR}. However, this approach does not allow to embed the a priori information on the impulse responses $b,c$ characterizing the forward model (\ref{def_mod_inf}).
\end{remark}

It is worth noting that the solution to \eqref{eq:tikhonv_problem} can be interpreted as the Maximum a Posteriori (MAP) estimator of the following Bayesian model: $b$ and $c$ are modeled as independent Gaussian random processes with zero mean and covariance functions $\lambda_b\Kcal_b$ and $ \lambda_c\Kcal_c$, respectively. Let $\eta$ be the deterministic vector of hyperparameters, i.e. the one containing $\lambda_b$, $\lambda_c$ and the hyperparameters characterizing $\Kcal_b$, $\Kcal_c$. The MAP estimator maximizes $\pb(\mathrm y^+,\vartheta|\mathrm y^-,\eta,\sigma^2)$ which denotes joint probability of $\mathrm y^+$ and $\vartheta$ given $\mathrm y^-$, $\eta$ and $\sigma^2$. Assuming that the past data $\mathrm y^-$ does not affect the a priori probability on $\vartheta$,  i.e., $\pb(\theta| { \mathrm y}^{-},\eta,\sigma^2)=\pb(\theta|\eta,\sigma^2)$, we have
\begin{align}
    &\pb(\mathrm y^+,\vartheta|\mathrm y^-,\eta,\sigma^2)=\pb(\mathrm y^+|\vartheta,\mathrm y^-,\eta,\sigma^2)\pb(\vartheta|\eta,\sigma^2)\nn\\ 
    & =\frac{1}{{\kappa}} e^{-\frac{1}{2}\left(\frac{1}{\lambda_b}\norm{b}_{\Hcal_b}^2+\frac{1}{\lambda_c}\norm{c}_{\Hcal_c}^2\right)}
    \prod_{t=1}^N\frac{1}{\sqrt{2\pi\sigma^2}}e^{-\frac{1}{2}\frac{\varepsilon_\vartheta(t)^2}{\sigma^2}}\label{eq:probabilityMAP}
\end{align}
where $\kappa$ is the normalization constant. Thus, the maximization of \eqref{eq:probabilityMAP} with respect to $\vartheta$ is equivalent to problem \eqref{eq:tikhonv_problem}. Kernel-based estimators, of both linear and nonlinear dynamic models, proposed in the literature, e.g. \cite{romeres2019derivative,dalla2021kernel,darwish2018prediction}, lead to a linear infinite dimensional Tikhonov regularization problem, i.e. the prediction error $\varepsilon_\vartheta$ depends linearly on $\vartheta$, whose existence of a solution is a well established result \citep{ARONSZAJN1950}. However, the same conclusion does not apply to problem \eqref{eq:tikhonv_problem} because it is a nonlinear infinite dimensional Tikhonov regularization problem, i.e. $\varepsilon_\vartheta$ in (\ref{eq:pred_err}) is not linear with respect to $\vartheta$. Moreover, the hyperparameters vector $\eta$ is learned from the data optimizing the marginal likelihood function \citep{RASMUSSEN_WILLIAMNS_2006}. Due to the nonlinear nature of problem \eqref{eq:tikhonv_problem}, such function does not admit an explicit form. The aim of the next sections is to show that problem \eqref{eq:tikhonv_problem} admits solution and thus it is also possible to derive a procedure which approximates the optimization of the marginal likelihood for finding the optimal hyperparameters.

\section{Existence of the solution}\label{sec:exist_sol}
In this section we prove that Problem \eqref{eq:tikhonv_problem} admits solution, showing that \eqref{eq:tikhonv_problem} is equivalent to minimizing a lower semi-continuous functional over a compact set.
\begin{definition}
    Let $\Hcal$ be an Hilbert space with inner product $\innerprod{\cdot}{\cdot}_\Hcal$. Given a sequence $c ^{(n)}\in \Hcal$, with $n\in \Nbb$, we say that $c ^{(n)}$ weakly converges to $c \in\Hcal$, and we denote it $c ^{(n)}\rightharpoonup c $, when
    \begin{align*}
        \lim_{n\to\infty}\innerprod{ v}{c ^{(n)}}_\Hcal=\innerprod{ v}{c }_\Hcal,\quad \forall v \in\Hcal.
    \end{align*}
\end{definition}
We define the Cartesian product of the two Hilbert spaces as $\Hcal=\Hcal_b\times\Hcal_c$, which is still an Hilbert Space  with inner product 
\begin{align*}
    \innerprod{(b _1,c _1)}{(b _2,c _2)}_{\Hcal}={\gamma_b}\innerprod{ b _1}{b _2}_{\Hcal_b}+{\gamma_c}\innerprod{ c_1}{c_2}_{\Hcal_c}. 
\end{align*}
\begin{definition}\citep{chen2018stability}
    Let $\Hcal$ be the RKHS induced by the kernel $\mathcal{K}$. $\mathcal{K}$ is said to be stable if $\Hcal \subset \ell^1$.
\end{definition}
 Notice that, all the kernels mentioned in the previous section are stable since their main diagonal decays exponentially to zero. 
Thus,  the necessary and sufficient condition in \cite{carmeli2006vector} for which a kernel is stable  is satisfied.

\begin{assumption}\label{ass.stab_ker}
    The kernel functions $\mathcal{K}_b,\mathcal{K}_c$ are stable.
\end{assumption}

\begin{lemma}\label{lem:compact_ball}
 Consider the balls
    \begin{align*}
        &\Bcal_b=\{b \in\Hcal_b:\norm{b }^2_{\Hcal_b}\leq r_b\}\\
        &\Bcal_c=\{c \in\Hcal_c:\norm{c }^2_{\Hcal_c}\leq r_c\}
    \end{align*}
    where $r_b,r_c>0$ are finite. Under Assumption \ref{ass.stab_ker} the Cartesian product $\Bcal_b\times \Bcal_c$ is compact with respect to the weak topology induced by $\Hcal=\Hcal_b\times \Hcal_c$.
\end{lemma}
\begin{pf} We prove that the two balls are closed with respect to the weak topology.  Since any RKHS is a reflexive space, we have that $\mathcal H_b$ is a reflexive space. By the Banach-Alaoglu Theorem \citep[3.17 page 222]{zeidler1995applied} it follows that $\Bcal_{b}$ is compact in the weak topology induced by $\Hcal_b$. In a similar way, it is possible to prove that $\Bcal_c$ is compact in the weak topology induced by $\Hcal_c$. By the Tychonoff Theorem \cite[page 234]{munkres2000topology},  the Cartesian product $\Bcal_b\times\Bcal_c$ is compact in the weak topology induced by $\Hcal_b\times\Hcal_c$. \hspace*{\fill} \hspace*{\fill} \qed  
\end{pf}
\begin{lemma}\label{prop:lower_semicon}
    Consider the RKHS $\mathcal S\subset \ell^2\times\ell^2$ equipped with the induced norm 
  \begin{align*}
       \norm{(b, c)}^2_{\mathcal S}=\norm{b}^2_{\Xcal_b^{-1}}+\norm{c}^2_{\Xcal_c^{-1}}
    \end{align*}
    where $\Xcal_b,\Xcal_c\!\in\!\Rbb^{\infty\times\infty}$ are strictly positive kernels (in the sense of Moore)\footnote{Recall that a kernel $\mathcal K\in\Rbb^{\infty \times \infty}$ is strictly positive if  for any set of reals $\{c_1,\, \ldots,\, c_n\}$, with $n\in \Nbb$, such that at least one element is different from zero, we have $\sum_{i,j=1}^n c_ic_j [\mathcal K]_{ij}>0$.} and  \mbox{$\norm{b}^2_{\Xcal_b^{-1}}=b^\top\Xcal_b^{-1}b$},  \mbox{$\norm{c}^2_{\Xcal_c^{-1}}=c^\top\Xcal_c^{-1}c$}  with $\Xcal_{b}^{-1}$ and $\Xcal_c^{-1}$ denoting the \fix{ inverse matrix\footnote{\fix{
    Note that, $\Xcal_b$ corresponds to the linear operator (which is a  bijection for the stable kernels used in system identification)  \alg{T\;:\; &\ell^\infty\rightarrow \mathrm{Range}[T]\nn\\
  & b\mapsto \Xcal_b b\nn} 
  where $\ell^\infty$ denotes the space of bounded sequences, while  $\Xcal_b^{-1}$ corresponds to the inverse operator of $T$ \alg{T^{-1}\;:\; & \mathrm{Range}[T]\rightarrow \ell^\infty\nn\\
  & v\mapsto \Xcal_b^{-1} v.\nn}}}} of $\Xcal_{b}$ and $\Xcal_c$, respectively.
    Let \begin{align*}
    \begin{array}{clcl}
        g:\ &{ \Scal}\to \Rbb\\
        &(b,c)\mapsto\norm{b}^2_{\Xcal_b^{-1}}+\norm{c}^2_{\Xcal_c^{-1}}+b^\top v_b+c^\top v_c,
    \end{array}
    \end{align*}
    with  $(v_b,v_c)$ such that $(\Xcal_bv_b,\Xcal_cv_c)\in\Scal$.
    Then, $g$ is lower-semicontinuous with respect to the weak topology induced by $\Scal$.
\end{lemma}
\begin{pf}
Let $w_b=\Xcal_b v_b$ and $w_c=\Xcal_c v_c$, $w=(w_b,w_c)$. Then, notice that { $g(b,c)=\norm {(b,c)}^2_{\Scal}+\innerprod{(b,c)}{w}_\Scal$.} Take a sequence { $( b ^{(n)}, c ^{(n)})\in \Scal$}, with $n\in\Nbb$, such that $( b ^{(n)}, c ^{(n)})\rightharpoonup( b , c )\in { \Scal}$. 
   Consider
    \begin{align*}
        \norm{( b ^{(n)}, c ^{(n)})-( b , c )}^2_{ \Scal}\geq 0,
    \end{align*}
which implies
\begin{align*}
        \norm{ ( b ^{(n)}, c ^{(n)})}^2_{ \Scal}+ \norm{( b , c )}^2_{ \Scal}- 2\innerprod{  ( b ^{(n)}, c ^{(n)})}{( b , c )}_{ \Scal}\geq 0,
    \end{align*}
    and hence
    \begin{align*}
        \norm{ ( b ^{(n)}, c ^{(n)})}^2_{ \Scal}\geq 2\innerprod{  ( b ^{(n)}, c ^{(n)})}{( b , c )}_{ \Scal}-\norm{( b , c )}^2_{ \Scal}.
    \end{align*}
    By taking the limit in both sides, and adding the term $\innerprod{(b^{(n)},c^{(n)})}{w}_{ \Scal}$, we obtain
    \begin{align*}
        &\lim_{n\to\infty}{\norm{ ( b ^{(n)}, c ^{(n)})}^2_{ \Scal}}\!+\!\innerprod{(b^{(n)},c^{(n)})}{\!\!w}_{ \Scal}\\
        &\geq\!\!\!\lim_{n\to\infty}\!2{\innerprod{  ( b ^{(n)}\!, c ^{(n)})}{\!\!( b, c )}_{ \Scal}}\!+\!\innerprod{(b^{(n)}\!,c^{(n)})}{\!\!w}_{\Scal}\!-\!\norm{( b , c )}^2_{ \Scal}\!,
    \end{align*}
      \begin{align*}
        \lim_{n\to\infty}&{\norm{ ( b ^{(n)}\!, c ^{(n)})}^2_{ \Scal}}+\innerprod{(b^{(n)}\!,c^{(n)})}{\!\!w}_{ \Scal}\\
        &\!\geq\!2{\innerprod{( b, c )}{( b, c )}}_{ \Scal}+\innerprod{(b,c)}{w}_{ \Scal}-\norm{( b , c )}^2_{ \Scal},
    \end{align*}
    and hence
         \begin{align*}
        \lim_{n\to\infty}&{\norm{ ( b ^{(n)}, c ^{(n)})}^2_{ \Scal}}+\innerprod{(b^{(n)},c^{(n)})}{\!\!w}_{ \Scal}\\
        &\geq \norm{( b , c )}^2_{ \Scal}+\innerprod{(b,c)}{w}_{ \Scal},
    \end{align*}  
    i.e. $g$ is lower-semicontinuous.\hspace*{\fill} \qed 
\end{pf}
 Before to prove the continuity of the functional in \eqref{eq:Vn}, we need the following two lemmas.
\begin{lemma}\label{lem:invers_def}
    Let
\begin{align}\label{eq:inv_c}
\begin{array}{cccc}
      f:\ &\ell^2 &\to &{\mathbb{S}}\\
          &c  &\mapsto &h 
\end{array}
\end{align}
be the operator for which 
\begin{align}\label{eq:hk}
        h_k=-c_k-\sum_{i=1}^{k-1}c_ih_{k-i},\quad k\in\Nbb.
    \end{align}
Define 
\begin{align*}
 H(z)=1+\sum_{k=1}^{\infty}h_kz^{-k},  \quad \ C(z)=1+\sum_{k=1}^{\infty}c_kz^{-k}.   
\end{align*} Then, $C(z)H(z)=1$. 
\end{lemma}

\begin{pf}
We define $c_0=h_0=1$. We have to prove that 
    \begin{align*}
        C(z)H(z):=\sum_{i,j=0}^{\infty}c_ih_jz^{-(i+j)}=1.
    \end{align*}
    Let define $k:=i+j$. Then it follows that $j=k-i$, and hence 
    \begin{align*}
        C(z)H(z)&=\sum_{k=0}^{\infty}\sum_{i=0}^{k}c_ih_{k-i}z^{-k}\\
        &=c_0h_0+\sum_{k=1}^\infty\left(c_0h_{k}z^{-k}+\sum_{i=1}^{k} c_ih_{k-i}z^{-k}\right)\\
           &=1+\sum_{k=1}^\infty\left(h_{k}+\sum_{i=1}^{k} c_ih_{k-i}\right)z^{-k}\\
            &=1+\sum_{k=1}^\infty\left(-{\sum_{i=1}^{k} c_ih_{k-i}}+{\sum_{i=1}^{k} c_ih_{k-i}}\right)z^{-k}=1
    \end{align*}
    where in the last equality we exploited the definition of $h_k$ in \eqref{eq:hk}. \hspace*{\fill} \qed
\end{pf}

\begin{lemma}\label{lem:ch_convergence}
    Let $c ^{(n)}\in\Hcal_c$,  with $n\in\Nbb$, such that $c ^{(n)}\rightharpoonup c \in\Hcal_c$. Let $h ^{(n)}=f(c ^{(n)})$ and $h =f(c )$, where $f(\cdot)$ has been defined in \eqref{eq:inv_c}. Then, for any $k\in\Nbb$
    \begin{subequations}
    \begin{align}&\lim_{n\to \infty}c_k^{(n)}=c_k,\\ &\lim_{n\to \infty}h_k^{(n)}=h_k\label{eq:hk_lim}.\end{align}
    \end{subequations}
    \end{lemma}
\begin{pf}
Since $c^{(n)}\rightharpoonup c$, we have that
\begin{align*}
    \lim_{n\to \infty}\innerprod{ c ^{(n)}}{v}_{\Hcal_c}=\innerprod{c}{v}_{\Hcal_c}, \quad \forall v \in\Hcal_c.
\end{align*}
We take $v =\Kcal_c\mathbf{e}_k\in\Hcal_c$, where $\mathbf{e}_k$ is the infinite dimensional canonical vector, whose $k$-th component is equal to $1$ and the remaining ones are equal to $0$. It follows that
\begin{align}\label{eq:lim}
    \lim_{n\to\infty}c_k^{(n)}&=\lim_{n\to\infty}\innerprod{ c ^{(n)}}{\Kcal_c\mathbf{e}_k}_{\Hcal_c}\nn\\
    &= \innerprod{ c}{\Kcal_c\mathbf{e}_k}_{\Hcal_c}=c_k,
\end{align}
 where $c^{(n)}_k$, with $n\in\Nbb$, $k\in\Nbb$, is the $k$-th element of the sequence $c^{(n)}$ and takes values in 
 $\Rbb$. Hence, the first claim holds. 
We prove the second claim by induction. For $k=1$, by Lemma \ref{lem:invers_def}, it follows that
\begin{align*}
    \lim_{n\to\infty}h_1^{(n)}=\lim_{n\to\infty}-c_1^{(n)}=-c_1=h_1
\end{align*}
where we exploited the fact that $c_1^{(n)}$ approaches $c_1$ as $n\to \infty$ by \eqref{eq:lim}.
For $k>1$, suppose that \begin{align*}
\lim_{n\to\infty}h_{i}^{(n)}=h_{i},\quad  i=1,\,\dots,\,k-1.\end{align*} We want to prove that $\lim_{n\to\infty}h_{k}^{(n)}=h_{k}$.
By  \eqref{eq:hk}, $h_k^{(n)}$ is defined as
\begin{align*}
    h_{k}^{(n)}=-\sum_{i=1}^{k}c^{(n)}_ih^{(n)}_{k-i}.
\end{align*}
 Notice that $h_k^{(n)}$ is a finite sum of terms obtained as product of a finite number of terms of the type $h_i^{(k)}$ with $i=1,\,\dots,\,k-1$,  $c_i^{(n)}$ with $i=1,\,\dots,\,k$, and the latter converge to finite values by \eqref{eq:lim}. 
We conclude that
    \begin{align*}
    \lim_{n\to\infty}h_{k}^{(n)}=-\sum_{i=1}^{k}c_ih_{k-i}=h_{k}
\end{align*}
where in the last equality we exploited \eqref{eq:hk}.
So, we can conclude that also \eqref{eq:hk_lim} holds.\hspace*{\fill} \qed 
\end{pf}
For the subsequent analysis we introduce the following assumption:
\begin{assumption}\label{ass.i.c.}
    There exists $M\in\mathbb{N}$ for which the initial conditions $\mathrm y^-$ and $\mathrm u^-$ are such that $y(s)=u(s)=0$, for all $s<-M$.
\end{assumption}
\begin{proposition}\label{prop:finite_sum}
   Under Assumption \ref{ass.i.c.}, the average squared prediction error $V_N(\vartheta )$ is continuous with respect to the weak topology induced by $\Hcal$.
\end{proposition}
\begin{pf}
    Let  
    \begin{align} \label{def_nu}
        \nu(t):=B(z)u(t)=\sum_{k=1}^{t+M}b_ku(t-k). 
    \end{align}
    Note that, $\nu(t)=0$ for $t\leq -M$. By \eqref{eq:pred_err} and \eqref{def_nu} we have that
\begin{align}\label{eq:eps_h_delta}
    \varepsilon_\vartheta(t)&=\sum_{k=0}^{t+M}h_k\left[y(t-k)-\nu(t-k)\right]\\
    &=\sum_{k=0}^{t+M}h_ky(t-k)-\sum_{k=0}^{t+M}\sum_{l=1}^{t-k+M}h_kb_lu(t-k-l).\nn
\end{align}
    Take $\vartheta^{(n)}=(b ^{(n)},c ^{(n)})\rightharpoonup (b ,c 
    )=\vartheta$. By Lemma \ref{lem:ch_convergence} $c_k^{(n)}\rightarrow c_k$, $ h_k^{(n)}\rightarrow h_k$ for any $k\in \Nbb$. By \eqref{eq:eps_h_delta}, $\varepsilon_{\vartheta^{(n)}}(t)$  is a finite sum of terms obtained as product of a finite number of terms as $b_k$'s and $h_k$'s, then
    \begin{align*}
        \lim_{n\to \infty}\varepsilon_{\vartheta^{(n)}}(t)=\varepsilon_{\vartheta}(t),\quad t=1,\dots,N
    \end{align*}
    which implies 
    \begin{align*}
        \lim_{n\to\infty}V_N(\vartheta ^{(n)})=V_N(\vartheta).
    \end{align*}\hspace*{\fill} \qed 
\end{pf}
\fix{It is worth noting that our identification approach is not able to guarantee that the estimated $C(z)$ is minimum-phase. However, this does not impact the overall analysis. In fact, the objective function in the Tikhonov regularization problem \eqref{eq:tikhonv_problem} remains well-defined even when $h_k$ diverges as $k \rightarrow \infty$. This is because, under Assumption 2, the functional $V_N(\vartheta)$, i.e.  the average squared prediction error \eqref{eq:Vn}, solely depends on $h_0, \ldots, h_{N+M}$, see \eqref{eq:eps_h_delta}.}

\begin{theorem}\label{th:thikhonov_solution}
Under Assumptions  \ref{ass.stab_ker} and \ref{ass.i.c.}, the Tikhonov regularization problem \eqref{eq:tikhonv_problem} admits solution.
\end{theorem}
\begin{comment}
\begin{theorem}\label{th:thikhonov_solution}
Let $\Kcal_b,\Kcal_c$ be two stable kernels. Assume that the initial conditions $\mathrm y^-$ and $\mathrm u^-$ are such that there exists $M>0$ for which $y(s)=u(s)=0$ $\forall s<-M$. Then, the Tikhonov regularization problem \eqref{eq:tikhonv_problem} admits solution.
\end{theorem}
\end{comment}
\begin{pf}
Recall that $\vartheta=(b,c)$. By Lemma \ref{prop:lower_semicon} (with $\Xcal_b=\gamma_b^{-1} \Kcal_b, \Xcal_c=\gamma_c^{-1}  \Kcal_c,v_b=0, v_c=0$) and Proposition \ref{prop:finite_sum} it follows that
    \begin{align}\label{eq:non_linear_sum}
    V_N(\vartheta)+{\gamma_b} \norm{b }^2_{\Hcal_b}+{\gamma_c} \norm{c }^2_{\Hcal_c}    \end{align}
    is lower-semicontinuous with respect to the weak topology induced by $\Hcal_b\times\Hcal_c$. Since $V_N(\vartheta)$ is bounded from below,  it follows that the functional in \eqref{eq:non_linear_sum} is coercive. Thus, there exists two balls $\Bcal_b=\{b \in\Hcal_b:\norm{b }^2_{\Hcal_b}\leq r_b\}$ and $\Bcal_c=\{c \in\Hcal_c:\norm{c }^2_{\Hcal_c}\leq r_c\}$ with $r_b$ and $r_c$ sufficiently large such that Problem \eqref{eq:tikhonv_problem} is equivalent to 
    \begin{align}\label{eq:tik_prob}
    \min_{(b ,c ) \in \Bcal_b\times \Bcal_c}{V_N(b,c )+{\gamma_b}\norm{b }^2_{\Hcal_b}+{\gamma_c}\norm{c }^2_{\Hcal_c}}.
\end{align}
By Lemma \ref{lem:compact_ball}, $\Bcal_b\times\Bcal_c$ is compact in the weak topology. Therefore, the Weierstrass theorem \citep[Theorem 2.D]{zeidler1995applied} guarantees the existence of a minimum for \eqref{eq:tik_prob}, and thus for \eqref{eq:tikhonv_problem}, because we minimize a lower-semicontinuous functional over a compact set. 
\hspace*{\fill} \qed 
\end{pf}

\fix{\begin{remark}
The theorem above is important because it determines whether  it is legitimate to approximate Problem \eqref{eq:tikhonv_problem} with its finite dimensional version, i.e., when $b$ and $c$ are replaced by their truncated version. Since Theorem \ref{th:thikhonov_solution} guarantees the existence of a solution, which belongs to $\Hcal_b\times\Hcal_c\subset \ell^1\times\ell^1$, the approximation is legitimate. In plain words, the bias introduced by the truncation is negligible, provided that the practical length of the truncated impulse responses is chosen sufficiently large.
In particular, the finite-dimensional approximation can be easily computed using the Matlab routine {\rm \texttt{armax.m}}.
\end{remark}}
\smallskip

We conclude this section providing an alternative proof of Theorem \ref{th:thikhonov_solution}. The latter provides some insights on the structure of such solution. We partition $b$ as follows
\begin{align*}
    b=\begin{bmatrix}
        \bar{b}\\
        \check{b}
    \end{bmatrix}, \quad 
    \bar{b}=\begin{bmatrix}
        b_1\\
        \vdots\\
        b_{N+M}
    \end{bmatrix}, \quad 
    \check{b}=\begin{bmatrix}
        b_{N+M+1}\\
        b_{N+M+2}\\
        \vdots
    \end{bmatrix}.
\end{align*}
and we apply the same partition on $\Kcal_b$ and its inverse 
\begin{align*}
    \Kcal_b=\begin{bmatrix}
        \Kcal_{\bar{b}} & \Kcal_{\bar{b}\check{b}}\\
        \Kcal_{\check{b}\bar{b}} & \Kcal_{\check{b}}
    \end{bmatrix},\quad  
    \Rcal_b=\Kcal_b^{-1}=\begin{bmatrix}
        \Rcal_{\bar{b}} & \Rcal_{\bar{b}\check{b}}\\
        \Rcal_{\check{b}\bar{b}} & \Rcal_{\check{b}}
    \end{bmatrix}
\end{align*}
where $\Kcal_{\bar{b}},\Rcal_{\bar{b}}\in \mathbb{R}^{(N+M)\times (N+M)},\ \Kcal_{\check{b}\bar{b}},\ \Rcal_{\check{b}\bar{b}}\in\mathbb{R}^{\infty\times{({N+M})}} $ and $\Kcal_{\check{b}},\Rcal_{\check{b}}\in \mathbb{R}^{\infty\times\infty}$. We perform the same partition on $c$, $\Kcal_c$ and its inverse \fix{matrix} using the corresponding notation.
Let $\bar{\vartheta}=(\bar{b},\bar{c})$ and $\check{\vartheta}=(\check{b},\check{c})$. Then, Problem \eqref{eq:tikhonv_problem} is equivalent to 
\begin{align}\label{eq:new_tikhonv_problem}
    &\min_{\substack{\bar{\vartheta}\in \fix{\mathbb{R}^{(N+M)}\times\mathbb{R}^{(N+M)}} \\ \check{\vartheta}\in \Hcal_{\check{b}}\times\Hcal_{\check{c}} }} J_1(\bar{\vartheta},\check{\vartheta})+J_2(\bar{\vartheta},\check{\vartheta})
\end{align}
where
\begin{align*}
    J_1(\bar{\vartheta},\check{\vartheta})&=V_N(b,c)+{\gamma_b}\norm{\bar{b}}^2_{\Rcal_{ \bar b}}+{\gamma_c}\norm{\bar{c}}^2_{\Rcal_{\bar c}},\\
    J_2(\bar{\vartheta},\check{\vartheta})&={\gamma_b}(\norm{\check{b}}^2_{\Rcal_{\check{b}}}+2\check{b}^\top\Rcal_{\check{b}\bar{b}}\bar{b} )+{\gamma_c}(\norm{\check{c}}^2_{\Rcal_{\check{c}}}+2\check{c}^\top\Rcal_{\check{c}\bar{c}}\bar{c} ),
\end{align*}
 $\Hcal_{\check{b}}$ and $\Hcal_{\check{c}}$ are  the RKHS
with  induced norm $\norm{\cdot}^2_{\Rcal_{\check{b}}}$ and $\norm{\cdot}^2_{\Rcal_{\check{c}}}$, respectively. 
By Assumption \ref{ass.i.c.} equation \eqref{eq:eps_h_delta} in the proof of Proposition \ref{prop:finite_sum} { holds. By \eqref{eq:eps_h_delta}, we have that $\varepsilon_{\vartheta}(t)$ only depends on $b_j$ with $j\in \{1\ldots t+M\}$ and $h_i$ with $i\in \{1\ldots t+M\}$. By Lemma \ref{lem:invers_def}, we have that $h_i$ depends on $c_{q}$ with ${q}\in \{1\ldots i\}$. Accordingly, $\varepsilon_{\vartheta}(t)$ only depends on $b_j$ with $j\in \{1\ldots t+M\}$ and $c_{q}$ with ${{q}}\in \{1\ldots t+M\}$. We conclude that}
$V_N$ does not depend on $\check{\vartheta}$. Thus, with some abuse of notation, we have
\begin{align*}
    V_N(\bar{\vartheta},\check{\vartheta})=V_N(\bar{\vartheta}),\\
    J_1(\bar{\vartheta},\check{\vartheta})=J_1(\bar{\vartheta}).
\end{align*}
Hence, Problem \eqref{eq:new_tikhonv_problem} is equivalent to 
\begin{align*}
    \min_{\bar{\vartheta}\in \fix{\mathbb{R}^{(N+M)}\times\mathbb{R}^{(N+M)}}}\left[
    J_1(\bar{\vartheta})+\min_{{\check{\vartheta}\in \Hcal_{\check{b}}\times\Hcal_{\check{c}} }} J_2(\bar{\vartheta},\check{\vartheta})\right]. 
\end{align*}
Consider Lemma \ref{prop:lower_semicon} with $\Xcal_b=(\gamma_b \Rcal_{\check{b}})^{-1},$ $\Xcal_c=(\gamma_c \Rcal_{\check{c}})^{-1}$, $v_b=2\gamma_b\Rcal_{\check{b}\bar{b}}\bar{b}$ and  $v_c=2\gamma_c\Rcal_{\check{c}\bar{c}}\bar{c}$. Condition $(\Xcal_b v_b,\Xcal_c v_c) \in \mathcal S$ follows from the fact that 
$$\norm{\Xcal_b v_b}^2_{\Xcal_{{b}}^{-1}}=4\gamma_b\bar b^\top \Rcal_{\bar b \check{b}} \Rcal_{\check{b}}^{-1}\Rcal_{\check{b}\bar b} \bar b \leq 4\gamma_b\norm{\bar{b}}^2_{\Rcal_{\bar{b}}}<\infty$$   where we exploited the fact that 
$$  \Rcal_{\bar b}-\Rcal_{\bar b \check{b}} \Rcal_{\check{b}}^{-1}\Rcal_{\check{b}\bar b}\geq 0. $$ In a similar way, we have that  
$\norm{\Xcal_c v_c}^2_{\Xcal_{{c}}^{-1}}<\infty$. Hence, 
$$\norm{(\Xcal_b v_b,\Xcal_c v_c)}^2_{\Scal}=\norm{\Xcal_b v_b}^2_{\Xcal_{{b}}^{-1}}+ \norm{\Xcal_c v_c}^2_{\Xcal_{{c}}^{-1}}<\infty.$$
 Thus, $J_2$ is lower-semicontinous with respect to $\check{\vartheta}$. Moreover,
\begin{equation*}
    J_2(\bar{\vartheta},\check{\vartheta})\geq -{\gamma_b} \norm{\bar{b}}^2_{\Rcal_{\bar{b}}} -{\gamma_c} \norm{\bar{c}}^2_{\Rcal_{\bar{c}}},
\end{equation*}
i.e. $J_2(\bar \vartheta,\cdot)$ bounded from below and it is coercive. Accordingly, $J_2$ admits a point of minimum with respect to $\check{\vartheta}$. It is not difficult to see that $J_2$ is strictly convex and its unique point of minimum is obtained imposing the stationarity conditions:
\begin{equation}\label{eq:stat_cond}
\begin{split}
    \check{b}=-\Rcal^{-1}_{\check{b}}\Rcal_{\check{b}\bar{b}}\bar{b}\\
    \check{c}=-\Rcal_{\check{c}}^{-1}\Rcal_{\check{c}\bar{c}}\bar{c}.
\end{split}
\end{equation}
Substituting \eqref{eq:stat_cond} in \eqref{eq:new_tikhonv_problem} we obtain 
\begin{align}\label{eq:rephrased_new_tikhonv_problem}
    \!\min_{\bar{\vartheta}\in \fix{\mathbb{R}^{(N+M)}\times\mathbb{R}^{(N+M)}}}\!V_N(\bar\vartheta)\!+\!{\gamma_b}\norm{\bar{b}}^2_{\Kcal^{-1}_{\bar{b}}}\!+\!{\gamma_c}\norm{\bar{c}}^2_{\Kcal_{\bar{c}}^{-1}}
\end{align}
where we exploited the fact that $\Rcal_b=\Kcal^{-1}_{ b}$ and thus 
\begin{align*}
    \Kcal_{\bar{b}}^{-1}=\Rcal_{\bar{b}}-\Rcal_{\bar{b}\check{b}}\Rcal^{-1}_{\check{b}}\Rcal_{\check{b}\bar{b}}\\
    \Kcal_{\bar{c}}^{-1}=\Rcal_{\bar{c}}-\Rcal_{\bar{c}\check{c}}\Rcal^{-1}_{\check{c}}\Rcal_{\check{c}\bar{c}}.
\end{align*}
Finally, the finite dimensional problem \eqref{eq:rephrased_new_tikhonv_problem} admits solution because the search of the minimum can be restricted to the compact set 
\begin{align*}
    \Bcal\!=\!\left\{\!(\bar{b},\bar{c})\!\in\! \fix{\mathbb{R}^{(N+M)}\!\times\!\mathbb{R}^{(N+M)}}\!: {\gamma_b}\norm{\bar{b}}^2_{\Kcal_{\bar {b}}^{-1}}\!+\!{\gamma_c}\norm{\bar{c}}^2_{\Kcal^{-1}_{ \bar c}}\!\leq\! r
    \!\right\}
\end{align*}
with $r>0$ taken sufficiently large and the objective function is continuous on $\Bcal$. We conclude that a solution to  Problem \ref{eq:tikhonv_problem} is
\begin{align}\label{eq:inf_dim_solution}
    b=\begin{bmatrix}
        \bar{b}\\
        \Kcal_{\check{b}\bar{b}}\Kcal_{\bar{b}}^{-1}\bar{b}
    \end{bmatrix}, \quad c=\begin{bmatrix}
        \bar{c}\\
         \Kcal_{\check{c}\bar{c}}\Kcal_{\bar{c}}^{-1}\bar{c}
    \end{bmatrix}
\end{align}
where $(\bar{b},\bar{c})$ is a solution to \eqref{eq:rephrased_new_tikhonv_problem} and we exploited the identities 
\begin{align*}
    \Rcal_{\check{b}\bar{b}}=-\Rcal_{\check{b}}^{-1}\Kcal_{\check{b}\bar{b}}\Kcal_{\bar{b}}^{-1}\\
    \Rcal_{\check{c}\bar{c}}=-\Rcal_{\check{c}}^{-1}\Kcal_{\check{c}\bar{c}}\Kcal_{\bar{c}}^{-1}.
\end{align*}
 In view of  (\ref{eq:inf_dim_solution}) we have the following corollary. 
\begin{comment}
\begin{corollary} Under the hypotheses of Theorem \ref{th:thikhonov_solution} the solutions of Problem (\ref{eq:tikhonv_problem}) admit the following representation 
$$  b=\sum_{k=1}^{N+M}\alpha_k [\mathcal K_b]_k, \quad c=\sum_{k=1}^{N+M}\beta_k [\mathcal K_c]_k, $$
where $[\mathcal K_b]_k$,  $[\mathcal K_c]_k$  denote the $k$-th column of $\mathcal K_b,\mathcal K_c$, respectively, and $\alpha_k,\beta_k$, with $k=1,\,\ldots,\, N+M$, are reals.
\end{corollary}
\end{comment}
\begin{corollary} Under Assumptions  \ref{ass.stab_ker} and \ref{ass.i.c.}, the solutions of Problem (\ref{eq:tikhonv_problem}) admits the following representation 
$$  b=\sum_{k=1}^{N+M}\alpha_k [\mathcal K_b]_k, \quad c=\sum_{k=1}^{N+M}\beta_k [\mathcal K_c]_k, $$
where $[\mathcal K_b]_k$,  $[\mathcal K_c]_k$  denote the $k$-th column of $\mathcal K_b,\mathcal K_c$, respectively, and $\alpha_k,\beta_k$, with $k=1,\,\ldots,\, N+M$, are reals.
\end{corollary}
This result can be understood as a generalization of the representer theorem for sequences \citep{PILLONETTO2024111347, kimeldorf1970correspondence} to a particular class of empirical risk functionals and for which $\varepsilon_\vartheta$ is not linear with respect to $\vartheta$.

The expression in \eqref{eq:inf_dim_solution} also suggests how to compute the solution to Problem \eqref{eq:tikhonv_problem}: it is sufficient to compute numerically the solution to Problem \eqref{eq:rephrased_new_tikhonv_problem}, which considers the truncated impulse responses $(\bar b,\bar c)$ of practical length $N+M$, and then the solution to (\ref{eq:tikhonv_problem}) is given by (\ref{eq:inf_dim_solution}). 

%In practice, we can approximate the impulse responses by truncating them. According to Theorem \ref{th:thikhonov_solution}, this is acceptable  because the coefficients of the estimated impulse responses decay to zero as 
%$k\to\infty$. In plain words, the bias introduced by the truncation is negligible, provided that the practical length of the truncated impulse responses is chosen sufficiently large.

\section{Hyperparameter estimation}\label{hyperparameter estimation}
In order to compute the MAP estimator, solution to \eqref{eq:tikhonv_problem}, an estimate of $\eta$ is needed.
 In the following, we propose a procedure to estimate these hyperparameters from the data by means of the marginal likelihood.
 As discussed at the end of the previous section, we can consider an approximation of the impulse responses which consist of a truncated version of themselves. Consider the truncated version of \eqref{def_mod_inf}
\begin{align}
    \label{def_mod}y(t)=\sum_{k=1}^{T} b_k u(t-k)+\sum_{k=0}^T c_k e(t-k),
\end{align}
where $ T\in\Nbb$, chosen sufficiently large, is called practical length. In plain words, we approximate \eqref{def_mod_inf} through an high-order  MAX model.  The truncated parameter vector is defined as $\theta=[
    b^\top \ c^\top
]^\top$
where, with some abuse of notation we use the same symbols used for the infinite dimensional case, to indicate $b^\top=[b_1\ b_2\ \dots\ b_{T}]$ and $c^\top=[c_1\ c_2\ \dots\ c_T]$.
Accordingly, $B(z)$, and $C(z)$ are, from now on, considered as finite sums; the one step ahead predictor \eqref{hatytheta} and the prediction error \eqref{eq:pred_err}, are consistent with these new definitions. We also introduce the finite dimensional kernel matrix as 
\begin{align*}
    K_\eta=\begin{bmatrix} 
       \lambda_b K_{b} & 0 \\ 0 &  \lambda_cK_{c} 
    \end{bmatrix}
\end{align*}
 where $K_b$ and $K_c$ are the top-left $T\times T$ sub-matrices of $\Kcal_b$ and $\Kcal_c$ respectively, and $\lambda_b,\lambda_c>0$ are the scale factors. Notice that, we made explicit its dependence on $\eta$, namely the hyperparameters that characterize the kernels. In this setup we have that $\theta$ is a Gaussian random vector with zero mean, and covariance matrix $K_\eta$, namely $\theta\sim \mathcal{N}(0,K_\eta)$.
It is straightforward to see that the finite dimensional Tikhonov regularization problem for $\theta$ is
\begin{align} \label{Theta_hat}
    {\hat{\theta}_{\eta}}= \argmin_{\theta \in \Rbb^{2T}}V_N(\theta)+ \frac{\sigma^2}{N}\norm{\theta}^2_{{K_\eta}^{-1}}.  
\end{align}
It is not difficult to see that
\begin{align}
\pb(\mathrm y^+|\theta,\mathrm y^-,\eta,\sigma^2) &=\prod_{t=1}^N\frac{1}{\sqrt{2\pi\sigma^2}}e^{\left(-\frac{1}{2\sigma^2} \varepsilon_\theta(t)^2 \right)}\nn,\\
\pb(\theta|\eta,\sigma^2)&=\frac{1}{\sqrt{(2\pi)^{2T}|{K_\eta}}|}e^{\left(-\frac{1}{2} \|\theta\|^2_{K_\eta^{-1}} \right)}\nn
\end{align}
where $|K_\eta|$ denotes the determinant of $K_\eta$. 
Taking into account the approximation \eqref{eq:probabilityMAP} where $\vartheta$ is now replaced with $\theta$, we have that
\begin{align}
\label{neg_lik}
    \ell(\theta,\eta;\mathrm y^+)&=-\log \pb(\mathrm y^+,\theta|\mathrm y^-,\eta,\sigma^2)\\
    &=\frac{N}{2\sigma^2}V_N(\theta)+\frac{1}{2}\norm{\theta}^2_{K^{-1}_\eta}+\frac{1}{2}\log{|K_\eta|}+\kappa\nn
\end{align}
where $\kappa$ is a constant not depending on $\theta$ and $\eta$. An estimate of $\sigma^2$ can be obtained using a low-bias ARX model as suggested in \cite{GOODWIN_1992}.
Then, an estimate of $\eta$ is given minimizing the negative log-marginal likelihood
\begin{align}\label{likelihood_minimization}
    \hat{\eta}=\argmin_{\eta\in \Ccal}\ell(\eta;\mathrm  y^+)
\end{align}
where $\ell(\eta;\mathrm y^+)=-\log \pb(\mathrm y^+|\mathrm y^-,\eta,\sigma^2)$, 
\begin{align}\label{integral}
    \pb(\mathrm y^+|{\mathrm y^-,}\eta,\sigma^2)=\int_{\Rbb^{2T}} \pb (\mathrm y^+,\theta|\mathrm y^-,\eta,\sigma^2)d\theta
\end{align}
and $\Ccal$ is the set of constraints for the hyperparameters.

The difficulty in the optimization problem \eqref{likelihood_minimization} is the evaluation of the objective function $\ell$ because it is difficult to find a closed form expression for the integral in \eqref{integral}. In what follows we propose a method to evaluate approximately the function $\ell(\eta;\mathrm y^+)$, given the hyperparameters vector $\eta$. We start by writing the integral \eqref{integral} in terms of its Laplace approximation \citep[Chapter 27]{mackay2003information}. Let 
\begin{align} \label{eq:h}
    h(\theta, \eta):=\frac{1}{2\sigma^2}V_N(\theta)+\frac{1}{2N}\norm{\theta}^2_{K^{-1}_\eta}+\frac{1}{2N}\log{|K_\eta|}.
\end{align}
Hence, in view of (\ref{neg_lik}) we have 
\begin{align}
&\ell(\theta,\eta;\mathrm y^+)=Nh(\theta, \eta)+\kappa,\nn\\
&\pb (\mathrm y^+,\theta|\mathrm y^-,\eta,\sigma^2)=\frac{1}{\tilde \kappa}e^{-Nh(\theta, \eta)}\nn
\end{align}
where $\tilde \kappa=e^{\kappa}$ is the normalization constant. Moreover,
\begin{align}\label{integral2}
     \pb(\mathrm y^+|\mathrm y^-,\eta,\sigma^2)&=\frac{1}{\tilde \kappa}\int_{\Rbb^{2T}}{}{e^{-Nh(\theta, \eta)}d\theta.}
   \end{align}
   The Laplace approximation method revolves around the idea to approximate $h(\theta;\eta)$ in (\ref{integral2}) with the second order Taylor expansion around an estimate of $\theta$ for which it holds $$\frac{\partial h(\theta,\eta)}{\partial\theta}=0.$$ In particular we choose to expand the series around a point of minimum of the function $h(\theta;\eta)$. The minimum is found by solving numerically the minimization problem in $\eqref{Theta_hat}$ for the given $\eta$. We denote this solution as ${\hat{\theta}_\eta}$ and then we can compute the integral using such approximation:
 \begin{align*}
    \pb( &\mathrm y^+|\mathrm y^-,\eta,\sigma^2)\\ &\simeq \frac{1}{\tilde \kappa}\left[e^{-Nh({\hat \theta_\eta}, \eta)}\left(\frac{2\pi}{N}\right)^{T}\sqrt{\left|\frac{\partial^2{h(\theta, \eta)}}{\partial{\theta}\partial\theta^\top}\Bigr|_{\theta=\hat{\theta}_{\eta}}\right|^{-1}}\right].\nn
\end{align*}
Then, $$\ell(\eta;\mathrm y^+)\approx \hat{\ell}(\eta; \mathrm y^+)+\check\kappa$$ where $\hat{\ell}(\eta ; \mathrm y^+ )$ is defined as
\begin{align}\label{lhat}
\hat{\ell}(\eta ; \mathrm y^+ ) := Nh( \hat{\theta}_{\eta},\eta)+\frac{1}{2}{{\log\left(\left| \frac{\partial^2{h(\theta, \eta)}}{\partial{\theta}\partial\theta^\top} \Bigr|_{\theta=\hat{\theta}_{\eta}}\right|\right)}}
\end{align}
and $\check \kappa$ is a constant not depending on $\eta$.
Next, we characterize the expression for $\hat\ell$ once $\hat\theta_\eta$ is given. First, we compute the derivatives of $h(\theta,\eta)$ defined in \eqref{eq:h} with respect to $\theta$:
\begin{align*}
    \frac{\partial{h(\theta, \eta)}}{\partial{\theta}}&=\frac{\partial}{\partial \theta}\left[\frac{1}{2\sigma^2}V_N(\theta)+\frac{1}{2N}\norm{\theta}^2_{{K_\eta}^{-1}}\right]\\
    &=\frac{1}{2\sigma^2}\frac{\partial V_N(\theta)}{\partial\theta}+\frac{1}{N}{K_\eta}^{-1}\theta\nn\\
    &=\frac{1}{N}\left[\frac{N}{2\sigma^2}\frac{\partial V_N(\theta)}{\partial\theta}+{K_\eta}^{-1}\theta\right];\nn
\end{align*}
\begin{align}\label{der_2_h}
    \frac{\partial^2{h(\theta, \eta)}}{\partial{\theta}\partial\theta^\top}=\frac{1}{N}\left[\frac{N}{2\sigma^2}\frac{\partial^2 V_N(\theta)}{\partial{\theta}\partial\theta^\top}+{K_\eta}^{-1}\right].
\end{align}
In regard to $V_N(\theta)$, its first order derivative with respect to $\theta$ is
\begin{align*}
  \frac{\partial V_N(\theta)}{\partial\theta}=\frac{1}{N}\sum_{t=1}^N  2\varepsilon_\theta(t)\frac{\partial\varepsilon_\theta(t)}{\partial\theta}=-\frac{2}{N}\sum_{t=1}^N\psi_\theta(t)\varepsilon_\theta(t)
\end{align*}
where
\begin{align*}
    \psi_\theta(t)=\frac{\partial \hat{y}_\theta(t|t-1)}{\partial\theta},
\end{align*}
and the second order derivative is 
\begin{align}\label{eq:sec_der_vn}
        \frac{\partial^2 V_N(\theta)}{\partial{\theta}\partial\theta^\top}=\frac{2}{N}\sum_{t=1}^N \left[ \psi_\theta(t)\psi^\top_\theta(t)-\varepsilon_\theta(t)\Gamma_{\theta}(t)\right]
\end{align}
where
\begin{align*}
    \Gamma_{\theta}(t)=\frac{\partial^2 \hat{y}_\theta(t|t-1)}{{\partial{\theta}\partial\theta^\top}}.
\end{align*}
Substituting \eqref{eq:sec_der_vn} in \eqref{der_2_h}, we get
\begin{align*}
    \frac{\partial^2h(\theta,\eta)}{{\partial{\theta}\partial\theta^\top}}\!=\!\frac{1}{N}\!\left[\frac{1}{\sigma^2}\!\sum_{t=1}^N\left(\psi_\theta(t)\psi_\theta^\top(t)-\varepsilon_\theta(t)\Gamma_{\theta}(t)\right)+K_\eta^{-1}\right]
\end{align*}
and thus
\begin{align*}
        \!&\!\hat{\ell}(\eta ; \mathrm y^+ )\!=\frac{N}{2\sigma^2}V_N({\hat{\theta}_\eta})+\frac{1}{2}||{\hat{\theta}_\eta}||^2_{K_\eta^{-1}}+\frac{1}{2}\log|K_\eta|\hspace{7.584px}\\
        \!&\!+\!\frac{1}{2}\log\!\left|\frac{1}{N}\!\left[\!\frac{1}{\sigma^2}\sum_{t=1}^N\!\left(\psi_{{\hat{\theta}_\eta}}\!(t)\psi^\top_{{\hat{\theta}_\eta}}\!(t)\!-\!\varepsilon_{{\hat{\theta}_\eta}}\!(t)\Gamma_{{\hat{\theta}_\eta}}(t)\right)\!+\!K_\eta^{-1}\right]\right|.
\end{align*}
It remains to be examined whether it is possible to compute $\psi_\theta(t)$ and $\Gamma_\theta(t)$ in an efficient way as their computation involves filtering operations. We start with the computation of { $\psi_\theta (t)\!=\!\left[\begin{smallmatrix}
        \psi_{\theta,B}^\top (t)\!&
        \psi_{\theta,C}^\top (t)      
    \end{smallmatrix}\right]^\top\in\Rbb^{2T}$}
where
\begin{align}
    \psi_{{\theta},B}(t)=\begin{bmatrix}
        \frac{\partial\hat{y}_{\theta}(t|t-1)}{ \partial b_1}\\
        \vdots\\
        \frac{\partial\hat{y}_{\theta}(t|t-1)}{\partial b_{T}}
    \end{bmatrix}, 
    \quad
     \psi_{{\theta},C}(t)=\begin{bmatrix}
        \frac{\partial\hat{y}_{\theta}(t|t-1)}{\partial c_1}\\
        \vdots\\
        \frac{\partial\hat{y}_{\theta}(t|t-1)}{\partial c_T}
    \end{bmatrix} \nn.
\end{align}
From \eqref{hatytheta} it is not difficult to see that
\begin{subequations}
    \begin{align}
            \frac{\partial\hat{y}_{\theta}(t|t-1)}{\partial b_k}&=\frac{1}{C(z)} u(t-k)\label{yhat_der_b},\\
        \frac{\partial\hat{y}_{\theta}(t|t-1)}{\partial c_k}
    &=\frac{1}{C(z)}\varepsilon_{\theta}(t-k)\label{yhat_der_c}.
    \end{align}
\end{subequations}
If we define
\begin{equation}\label{eq:filter1}
    \tilde{u}(t)=\frac{1}{C(z)}u(t), \quad \tilde{\varepsilon}_{\theta}(t)=\frac{1}{C(z)}\varepsilon_{\theta}(t),
\end{equation}
then, we can write
\begin{align}
   \psi_{{\theta},B}(t)=\begin{bmatrix}
        \tilde{u}(t-1)\\
        \vdots\\
       \tilde{u}(t-T)
    \end{bmatrix},  \quad  \psi_{{\theta},C}(t)=\begin{bmatrix}
        \tilde{\varepsilon}_{\theta}(t-1)\\
        \vdots\\
      \tilde{\varepsilon}_{\theta}(t-T)
    \end{bmatrix}. \nn
\end{align}
Consider now the computation of $\Gamma_{{\theta}}(t)\in \Rbb^{2T\times 2T}$.
It is not difficult to see
\begin{subequations}
    \begin{align}
        \frac{\partial^2\hat{y}_{\theta}(t|t-1)}{\partial b_k\partial b_{\Tilde{k}}}&=0,\\
        \frac{\partial^2 \hat{y}_{\theta}(t|t-1) }{\partial b_k \partial c_{\Tilde{k}} }&=-\frac{1}{C(z)}\frac{\partial \hat{y}_{\theta}(t-\Tilde{k}|t-\Tilde{k}-1) }{\partial b_k }\label{double der bc}.
    \end{align}
\end{subequations}
Substituting \eqref{yhat_der_b} and \eqref{eq:filter1} in \eqref{double der bc} we get
\begin{align*}
      \frac{\partial^2 \hat{y}_{\theta}(t|t-1) }{\partial b_k \partial c_{\Tilde{k}} }=-\frac{1}{{C(z)}} \tilde u(t-(k+\Tilde{k})).
\end{align*}
As it concerns the remaining partial derivative, we have that
\begin{align} \label{double c der}
      \frac{\partial^2\hat{y}_\theta(t|t-1)}{\partial{c_k\partial{c_{\Tilde{k}}}}}= -\frac{1}{C(z)}&\left[\frac{\partial{\hat{y}_{\theta}(t-k|t-k-1)}}{\partial c_{\Tilde{k}}}\right.\\
     &\ \left.+\frac{\partial\hat{y}_{\theta}(t-\Tilde{k}|t-\Tilde{k}-1)}{\partial{c_k}}\right].\nn
\end{align}
Substituting \eqref{yhat_der_c} and \eqref{eq:filter1} in \eqref{double c der}, we obtain
\begin{align*}
\frac{\partial^2\hat{y}_{\theta}(t|t-1)}{\partial{c_k\partial{c_{\Tilde{k}}}}}=
     -\frac{2}{{C(z)}}\tilde\varepsilon_{\theta}(t-(k+\Tilde{k})).
\end{align*}
Notice that the matrix $\Gamma_\theta(t)$ can be written as
\begin{align*}
  \Gamma_{\theta}(t)=\begin{bmatrix}
  0_{T \times T} & \Gamma_{{\theta},B}(t)\\
  {{\Gamma_{{\theta},B}}(t)}^\top & \Gamma_{{\theta},C}(t)
\end{bmatrix}  
\end{align*}
where
\begin{align*}
        \Gamma_{{\theta},B}(t)&=\begin{bmatrix}
        \frac{\partial^2\hat{y}_{\theta}(t|t-1)}{\partial b_1\partial c_1} & \dots & \frac{\partial^2 \hat{y}_{\theta}(t|t-1)}{\partial b_1\partial c_T}\\
        \vdots&\ddots & \vdots\\
        \frac{\partial^2 \hat{y}_{\theta}(t|t-1)}{\partial b_{T}\partial c_1} & \dots & \frac{\partial^2 \hat{y}_{\theta}(t|t-1)}{\partial b_{T} \partial c_T} \end{bmatrix},\nn\\
        \Gamma_{{\theta},C}(t)&=\begin{bmatrix}
        \frac{\partial^2\hat{y}_{\theta}(t|t-1)}{\partial c_1\partial c_1}&\dots&\frac{\partial^2 \hat{y}_{\theta}(t|t-1)}{\partial c_1\partial c_T}\\
        \vdots &\ddots  & \vdots\\
        \frac{\partial^2\hat{y}_{\theta}(t|t-1)}{\partial c_T\partial c_1}&\dots&\frac{\partial^2 \hat{y}_{\theta}(t|t-1)}{\partial c_T \partial c_T}
    \end{bmatrix}.\nn\end{align*}
If we define
\begin{align}\label{eq.filter2}
    \bar{u}(t)=-\frac{1}{C(z)}\tilde u(t), \ \quad \bar{\varepsilon}_{\theta}(t)=-\frac{2}{C(z)}\tilde\varepsilon_{\theta}(t),
\end{align}
we obtain the \fix{Hankel} matrices
\begin{subequations}\label{eq:gamma}
\begin{align*}
       \Gamma_{{\theta},B}(t)&=
    \left[\begin{smallmatrix}
        \bar{u}(t-2) &\bar{u}(t-3) & \dots & \bar{u}(t-(T+1))\\
        \bar{u}(t-3) &\bar{u}(t-4) & \dots & \bar{u}(t-(T+2))\\
        \vdots& \vdots & &\vdots\\
        \bar{u}(t-(T+1)) &\bar{u}(t-(T+2))& \dots & \bar{u}(t-2T)\\
    \end{smallmatrix}\right],\\
    \Gamma_{{\theta},C}(t)&=
     \left[\begin{smallmatrix}
        \bar{\varepsilon}_{\theta}(t-2) &\bar{\varepsilon}_{\theta}(t-3) & \dots & \bar{\varepsilon}_{\theta}(t-(T+1))\\
        \bar{\varepsilon}_{\theta}(t-3) &\bar{\varepsilon}_{\theta}(t-4) & \dots & \bar{\varepsilon}_{\theta}(t-(T+2))\\
        \vdots& \vdots & &\vdots\\
        \bar{\varepsilon}_{\theta}(t-(T+1)) &\bar{\varepsilon}_\theta(t-(T+2))& \dots & \bar{\varepsilon}_{\theta}(t-2T)\\
    \end{smallmatrix}\right].
\end{align*}
\end{subequations}
Accordingly, the construction of $\psi_\theta(t)\in\Rbb^{2T}$ and $\Gamma_\theta(t)\in\Rbb^{2T\times 2T}$ for $t=1,\,\dots,\, N$ can be done performing only $4N$ (scalar) filtering operations to define $\tilde{u}(t)$, $\tilde{\varepsilon}_{\hat{\theta}_\eta}(t)$, $\bar u(t)$, $\bar{\varepsilon}_{\hat{\theta}_\eta}(t)$. Moreover, this number of operations does not depend on the chosen practical length $T$.
   For the sake of completeness, Algorithm \ref{Alg} reports the complete procedure for the evaluation of $\hat{\ell}(\eta ; \mathrm y^+)$ given $\eta$, $\sigma^2$, $\mathrm{u}^+$ and $\mathrm{y}^+$. The optimization problem in Step \ref{step1} is solved numerically using the \texttt{armax.m} routine of the Matlab System Identification Toolbox.
It is worth noting that the initial conditions are never completely known. To perform the previous filtering operations we set the initial conditions equal to zero. 
Such a choice introduces an error that goes to zero as $N$ increases, see Section 3.2 in \cite{LJUNG_SYS_ID_1999}.

        \begin{algorithm}
    \caption{Evaluation of $  \hat{\ell}(\eta ; \mathrm y^+ )$}\label{Alg}
     \textbf{Input}: $ \mathrm u^+$, $ \mathrm y^+$, $\eta$, $\sigma^2$.\\
  \textbf{Output}: $\hat{\ell}(\eta ; \mathrm y^+ )$
  
    \begin{algorithmic}[1]
      \State \label{step1}Compute $\hat{\theta}_\eta$ as solution to \eqref{Theta_hat};
      \State Compute $V_N(\hat\theta_\eta)$ and construct $K_\eta$;
      \State \mbox{Compute~the~prediction data} $\varepsilon_{\hat\theta_\eta}^+\!=~\![\,\varepsilon_{\hat\theta_\eta}\!(1)\,\ldots\,\varepsilon_{\hat\theta_\eta}\!(N)\,]^\top$ using $\mathrm y^+$, $ \mathrm u^+$;
      \State \mbox{Compute $ \mathrm{\tilde u^+}\!=\![\,\tilde u(1)\,\ldots\,\tilde u(N)\,]^\top$}, $\tilde\varepsilon_{\hat\theta_\eta}^+=~[\,\tilde\varepsilon_{\hat\theta_\eta}(1)\,\ldots\,\tilde\varepsilon_{\hat\theta_\eta}(N)\,]^\top$ through \eqref{eq:filter1} and using $ \mathrm u^+$, $\varepsilon_{\hat\theta_\eta}^+$;
      \State Construct $\psi_{\hat\theta_\eta}(t)$, $t=1,\,\dots,\,N$ using $ \mathrm{\tilde u^+}$, $\tilde\varepsilon_{\hat\theta_\eta}^+$;
      \State \mbox{Compute $ \mathrm{\bar{u}^+}=[\,\bar u(1)\,\ldots\,\bar u(N)\,]^\top$}, \mbox{$\bar\varepsilon_{\hat\theta_\eta}^+=[\,\bar\varepsilon_{\hat\theta_\eta}(1)\,\ldots\,\bar\varepsilon_{\hat\theta_\eta}(N)\,]^\top$} through \eqref{eq.filter2} and using $\tilde {\mathrm u}^+$, $\tilde\varepsilon_{\hat\theta_\eta}^+$;
      \State Construct $\Gamma_{\hat\theta_\eta}(t)$, $t=1,\,\dots,\,N$ leveraging on the \fix{Hankel} structure of the blocks and using $ \mathrm{\bar u^+}$, $\bar\varepsilon_{\hat\theta_\eta}^+$;
        \State Compute $\hat{\ell}(\eta ; \mathrm y^+)$ through (\ref{lhat}).
    \end{algorithmic}
    \end{algorithm}

    Finally, the estimator of ${\eta}$ is obtained by minimizing the Laplace approximation of $\ell$:
\begin{align}\label{hat_eta}
     \hat{\eta}=\argmin_{\eta\in \Ccal} \hat{\ell}(\eta ; \mathrm y^+).
\end{align}
This optimization problem is solved numerically using the \texttt{fmincon.m} Matlab routine.

As already mentioned, the value of $\sigma^2$ used in the evaluation of the approximated marginal likelihood can be estimated through a low-bias ARX model. 
An alternative strategy that can be considered is to embed $\sigma^2$ in the set of hyperparameters to be optimized by solving \eqref{hat_eta}. In such a case, the estimate obtained through the low-bias ARX model can still be used as a guess of $\sigma^2$ in the initialization of the algorithm. 

Finally, we highlight that, once the  hyperparameters have been estimated with the outlined procedure, an estimate $\hat{\theta}_{\hat{\eta}}$ of the impulse responses of \eqref{def_mod_inf}  is obtained by solving the optimization problem \eqref{Theta_hat} for $\eta=\hat{\eta}$. 

\textit{Version with pseudo-linear regression. 
}\label{Pseudo_linear_reg}

The previous procedure for the evaluation of the approximated marginal-likelihood requires an optimization operation for computing $\hat{\theta}_\eta$. Next, we propose a strategy which avoids the optimization operation and relies on the pseudo-linear regression approach \citep[Section 7.5]{LJUNG_SYS_ID_1999}. By the forward model \eqref{def_mod}, we have 
\begin{align}\label{eq:pseudolin}
        &y(t)=\underbrace{\sum_{k=1}^{T}b_ku(t-k)+\sum_{k=1}^{T}c_k\varepsilon_\theta(t-k)}_{=:\hat{y}_\theta(t)}+e(t)\\
        &=\underbrace{\begin{bmatrix}
            u(t-1)&u(t-2)&\dots&\varepsilon_\theta(t-1)&\varepsilon_\theta(t-2)&\dots\end{bmatrix}}_{=:\phi(t,\varepsilon_\theta(t))^\top}\theta+e(t)\nn
\end{align}
where we exploited the fact that\fix{, when $\theta$ is the true parameter vector,}
\begin{align*}
\varepsilon_\theta(t)=y(t)-\hat{y}_\theta(t)=e(t).
\end{align*}
The idea is to consider as estimate $\hat\varepsilon$ of $\varepsilon_\theta$, the one obtained from the low-bias ARX model already used for the estimation of the noise variance $\sigma^2$. Then, substituting $\hat\varepsilon$ in \eqref{eq:pseudolin}, we obtain \fix{the approximation}
\begin{align}\label{eq:lin reg y}
    y(t)\fix{\approx}\phi(t,\hat\varepsilon(t))^\top\theta+e(t),
\end{align}
which is a linear regression model with respect to $\theta$.

Since $\theta\sim\mathcal{N}(0,K_\eta)$, the preliminary estimate $\hat\theta_\eta$ of $\theta$ according to \eqref{eq:lin reg y} is given by 
\begin{align}\label{pseudo_lin_Theta_hat}
    \hat\theta_\eta=\argmin_\theta \sum_{t=1}^{N} \left(y(t)-{\phi(t,\hat\varepsilon(t))}^\top\theta\right)^2+\frac{\sigma^2}{N}\norm{\theta}^2_{K_\eta^{-1}}
\end{align}
which can be computed efficiently in closed form \citep{CHEN20132213}.
The complete procedure to evaluate $\hat\ell$ is outlined in Algorithm \ref{Alg3}.
  
\begin{algorithm}
    \caption{Evaluation of $  \hat{\ell}(\eta ; \mathrm y^+ )$}\label{Alg3}
     \textbf{Input}: $ \mathrm u^+$, $ \mathrm y^+$, $\eta$, $\sigma^2$, $\hat\varepsilon=[\,\hat{\varepsilon}(1),\,\dots,\,\hat{\varepsilon}(N)\,]^\top$.\\
  \textbf{Output}: $\hat{\ell}(\eta ; \mathrm y^+)$
 \begin{algorithmic}[1]
 \State Build the vector $\phi(t,\hat\varepsilon(t))$, $t\!=\!1,\,\dots,\,N$ using $ \mathrm u^+$, $\hat{\varepsilon}$;
 \State Compute $\hat{\theta}_\eta$ as solution to \eqref{pseudo_lin_Theta_hat};
      \State Compute $V_N(\hat\theta_\eta)$ and construct $K_\eta$;
      \State Perform Step 3 - Step 7 of Algorithm \ref{Alg};
        \State Compute $\hat{\ell}(\eta ; \mathrm y^+)$ through (\ref{lhat}).
    \end{algorithmic}
    \end{algorithm}
\section{Numerical Experiments}\label{sec:numres}
We present some case studies to evaluate the performance of the proposed estimator. We compare the performance of our estimator with the ones available in the literature.
More precisely, we considered {four} different experimental setups. For each of them, we generated {$N_m$} SISO discrete time random models. Then, we fed the generated models with different input signals and collected a total of $N$ samples for the training dataset $\Dcal:=\{(y(t),u(t)),\; t=1,\,\dots,\, N\}$.
Subsequently, each model was fed with unit variance WGN as input, and {$N_v$} samples were collected to form the validation dataset $\Dcal_v:=\{(y_v(t),u_v(t)),\; t=1,\,\dots,\,  N_v\}$. {In all the simulations, {if not differently specified}, we set  $N_m=200, N=2000, N_v=1000$ and the noise variance was chosen to have a Signal to Noise Ratio (SNR) equal to $1$.} The models and the input signals chosen in each setup are described below:
\begin{itemize}
    \item \textbf{Setup 1}: The models, of order 40, were created through the \texttt{drss.m} Matlab routine and such that the absolute value of their dominant pole belongs to the interval $[ 0.8, 0.9]$. The input signal, for the training dataset is a square wave whose period is equal to $650$ samples, with $50\%$ duty cycle and range between $[-0.5,0.5]$; 
    \item \textbf{Setup 2}: The models are as in Setup 1. The difference is that the input signal for the training dataset is WGN with unit variance;
    \item \textbf{Setup  3}: The models were generated such that their impulse responses $b$ and $c$ are as follows. Let $\mathcal{U}[\beta,\gamma]$ denote the uniform probability distribution over the real  interval $[\beta,\gamma]$. The impulse response $b$ is such that  $b = g_1 + g_2$ where
\begin{align*}
    g_{1,t} = \alpha_1\beta_1^t+ \alpha_2\beta_2^t,
\end{align*}
 with $\alpha_1 \sim \mathcal{U}[0, 1]$, $\alpha_2 \sim \mathcal{U}[8, 10]$, $\beta_1 \sim \mathcal{U}[0.88, 0.9]$, $\beta_2 \sim \mathcal{U}[0.3, 0.35]$; $g_2$ is the impulse response of a \fix{SISO system} of order 30 generated randomly with most of the poles in high frequency: $95\%$ of the poles are complex and the remaining ones are real; the $80\%$ of the complex ones have phase within the interval $[\pi/4,\pi/4+\pi/6]$, moreover $\norm{g_2}/\norm{g_1} = 0.05$.
Roughly speaking, those systems are characterized by an impulse response having a fast decay rate and a slow decay rate. The impulse response $c$, instead, is generated through the \texttt{drss.m} Matlab routine with $40$-th order and such that the absolute value of its dominant poles belongs to the interval $[ 0.8, 0.9]$. 
The input signal is a square wave whose period is equal to $40$ samples, with $50\%$ duty cycle and range between $[-1,1]$.
\item \textbf{Setup 4:} The Monte Carlo study is composed by $N_m=100$ experiments. In each experiment, we generated a SISO model using the procedure utilized in subsection  ``Benchmark Problems'' of \cite{pillonetto2023full} where the order of the polynomials is set equal to 40 and the input is chosen in the same way.  We collected $N=2000$ and $N_v=500$ samples for the training and test datasets, respectively.

\end{itemize}
In each case study, we considered one of the aforementioned setup and for each trial we estimated the model using the corresponding training dataset and the different identification paradigms. The practical length has been chosen as $T=50$.
Then, we tested the performance of each estimator,using two different indices: the Coefficient Of Determination (COD) and the Average Impulse Response (\textrm{AIR}) fit. The COD index for the $\mathsf{k}$-step ahead predictor is defined as:
\begin{align*}
 \textrm{COD}_\mathsf{k}=100\left(1-\frac{||y_v-\hat{y}_{\mathsf{k},v}||^2}{||y_v-\bar y_v||^2}\right),
\end{align*}
where $y_v=[\, y_v(1) \; y_v(2)\;\ldots\; y_v(\fix{N_v})\,]^\top$ is the output of the validation dataset, $\hat{y}_{\mathsf{k},v}$ is the vector containing the $\mathsf{k}$-step ahead predictions of $y_v$ using the estimated model and 
\begin{align*}
    \bar y_v=\frac{1}{N}\sum_{k=1}^N y_v(t).
\end{align*}
Let $\textrm{COD}_\mathsf{k}^i$ be the coefficient of determination of the i-th model computed using the $\mathsf{k}$-step ahead predictor. We define the average COD index for the $\mathsf{k}$-step ahead predictor over the 200 realizations as
\begin{align*}
    \overline{\textrm{COD}}_\mathsf{k}=\frac{1}{200}\sum_{i=1}^{200}\textrm{COD}_\mathsf{k}^i. 
\end{align*}
\fix{The AIR fit of the impulse response $b$ is defined as}
% The AIR fit of the estimator is defined as
% \begin{align*}
%    AIR=\frac{1}{2}(AIR(b,\hat{b})+AIR(c,\hat{c})),
% \end{align*}
% where
\begin{align*}
\fix{\rm AIR_b}=100\left(1-{\frac{\sum_{k=1}^{T}\left(b_k-\hat{b}_k\right)^2}{\sum_{k=1}^{T}\left(b_k-\bar{b}\right)^2}}\right),
\end{align*}
\fix{where} $b$ and $\hat b$ are the true and estimated impulse responses, respectively,
\begin{align*}
    \bar b=\frac{1}{T}\sum_{k=1}^Tb_k
\end{align*}
and \fix{the {\rm AIR} of $c$} is defined likewise.

\subsection{Case study I}
In this first study, we consider Setup 1 with the following estimators:
\begin{itemize}
    \item 
   \fix{ \textbf{FIR OE-TC} which denotes the Finite Impulse Response Output Error kernel-based estimator based on the parametrization (\ref{par_inf_ARX}) proposed in \cite{PILLONETTO_2011_PREDICTION_ERROR} where $a_k=0$ for all $k\in \mathbb N$ and equipped with the TC kernel defined in \eqref{defTC};} %TOGLIEREI\footnote{These estimators are ARX-based estimators because model (\ref{par_inf_ARX}) is approximated through a high-order ARX model.\label{ff}} 
    \item \textbf{PEM+BIC} which denotes classical PEM approach to estimate Box-Jenkins (BJ) models, as implemented in the \texttt{bj.m} function of the MATLAB System Identification Toolbox, of the form 
\begin{align*}
    y(t)=\frac{M(z)}{N(z)}u(t)+\frac{P(z)}{Q(z)}e(t)
\end{align*}
 with
\begin{align*}
M(z)&=\sum_{k=1}^{k_1} m_k z^{-k},\quad N(z)=1+\sum_{k=1}^{k_1} n_kz^{-k} \\
P(z)&=1+\sum_{k=1}^{k_2} p_k z^{-k},\quad Q(z)=1+\sum_{k=1}^{k_2} q_kz^{-k} \end{align*}
where the orders $k_1,k_2\in\{1,2,\dots, 10\}$. The model orders $k_1$ and $k_2$ are chosen according to the BIC criterion, see \cite{SCHWARZ_1978};
\item \textbf{MAX-TC} which denotes the kernel-based estimator proposed in Section~\ref{hyperparameter estimation} equipped with the TC kernel; the approximation of the marginal likelihood is performed according to Algorithm \ref{Alg};
\item \textbf{ARX-TC} denotes the kernel-based estimator based on the parametrization (\ref{par_inf_ARX}) proposed in \cite{PILLONETTO_2011_PREDICTION_ERROR} and equipped with the TC kernel;
\item \textbf{MAX-D2} which denotes the kernel-based estimator proposed in Section~\ref{hyperparameter estimation} equipped with the D2 kernel defined in \eqref{dDC2}; the approximation of the marginal likelihood is performed according to Algorithm \ref{Alg};
\item \textbf{ARX-D2} denotes the kernel-based estimator based on the parametrization (\ref{par_inf_ARX}) proposed in \cite{PILLONETTO_2011_PREDICTION_ERROR} and equipped with the D2 kernel.
\end{itemize}
In all the aforementioned \fix{kernel-based} estimators the noise variance is estimated with a low-bias model \citep{GOODWIN_1992}.
Figure~\ref{studio1} shows the boxplots relative to \fix{$\rm{AIR}_b$ (left panel)  and $\rm{AIR}_c$ (right panel) for the considered estimators. Note that, $\rm{AIR}_c$ is not depicted  for \textbf{FIR OE-TC} since the latter is only characterized by the impulse response $b$}. It is possible to notice that both the estimators \textbf{MAX-TC} and \textbf{MAX-D2} perform better than the ARX-based estimators using the same kernels (i.e. \textbf{ARX-TC}, \textbf{ARX-D2}). Furthermore, all these estimators outperform the \fix{$\mathrm{AIR}_b$ index of \textbf{FIR OE-TC} and \textbf{PEM+BIC} estimators while the $\mathrm{AIR}_c$ index of \textbf{PEM+BIC} estimator results to be the higher.} In Figure~\ref{studio_1_mean} it is plotted the average COD fit $\overline{ \textrm{COD}}_{\mathsf{k}}$ for $\mathsf{k}=1,\dots,12$ for each estimator. As we may expect, $\overline{ \textrm{COD}}_{\mathsf{k}}$ decreases as the \fix{prediction horizon} increases\footnote{\fix{For \textbf{FIR OE-TC}, the k-step-ahead predictor $\hat y(t|t-k)$ coincides with the one-step-ahead predictor $\hat y(t|t-1)$, since it depends only on the known input sequence. As a result, the COD index is independent of the prediction horizon $k$. 
}}. Both the \textbf{MAX-TC} and \textbf{MAX-D2} estimators exhibit an higher fit value with respect to all the other estimators for any  considered perdition horizon.

\begin{figure}
    \centering
 \includegraphics[width=0.92\linewidth]{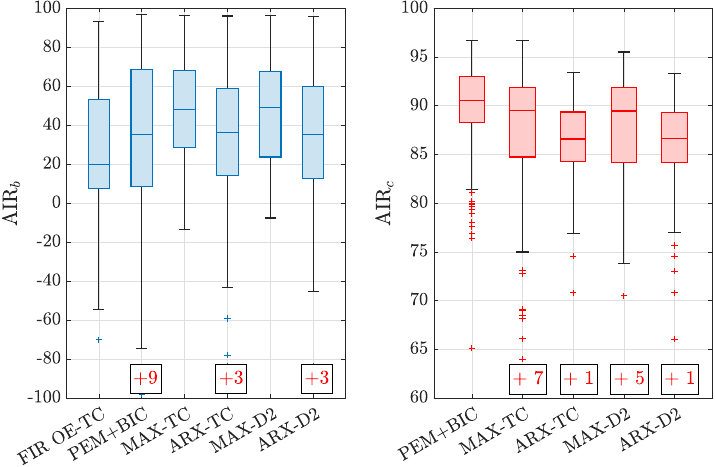}
\caption{\textit{Case-study I:} \fix{ boxplots of the \textrm{AIR} index of the plant ($\rm{AIR}_b$, left panel)  and of the error model ($\rm{AIR}_c$, right panel) over 200 simulations for the considered estimators.} The number in the box denotes the amount of outliers not shown in the picture.}
    \label{studio1}
\end{figure}

\begin{figure}
    \centering
\includegraphics[width=0.47\textwidth]{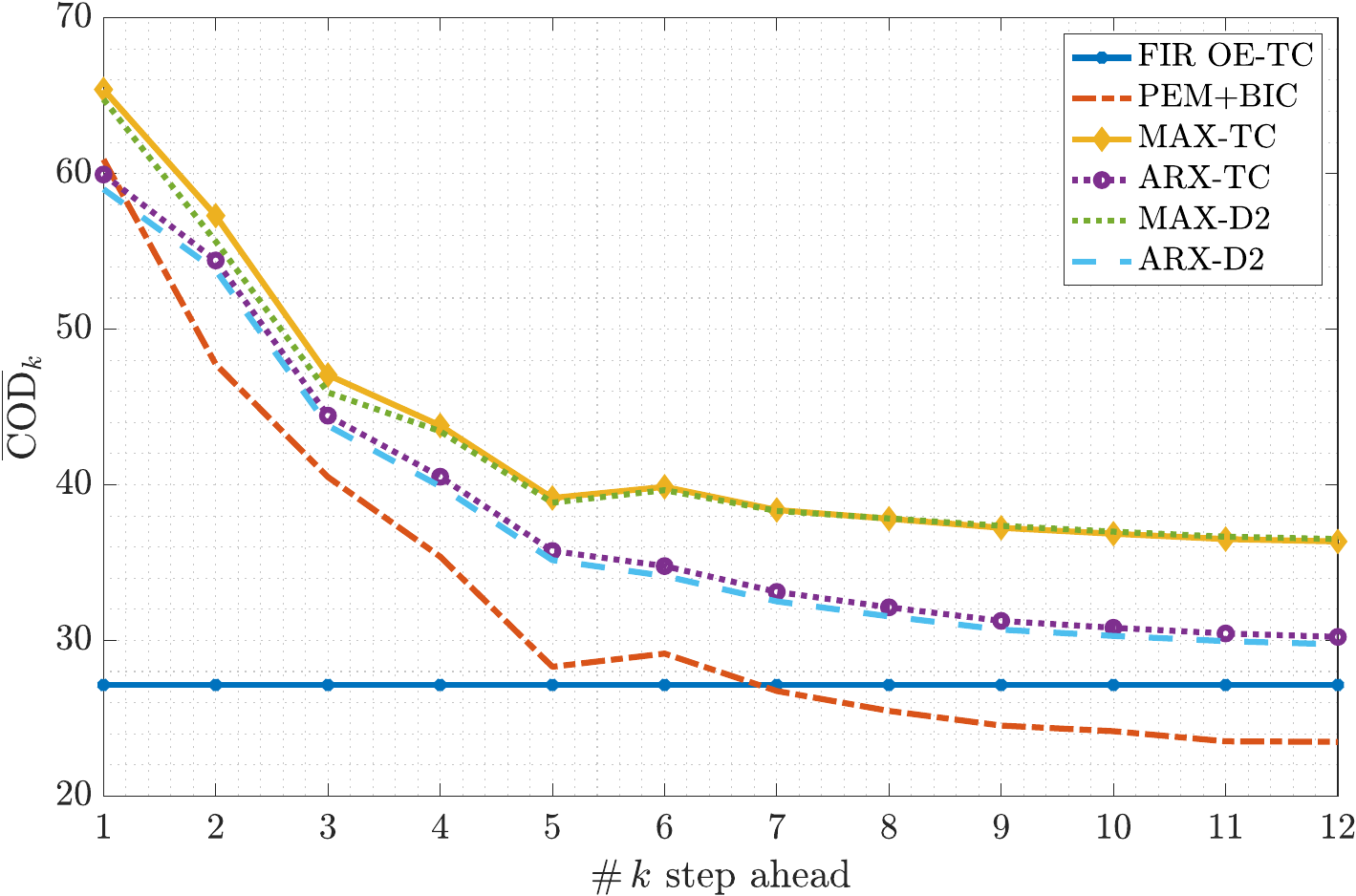}
     \caption{\textit{Case study I:} average COD  index $\overline{ \textrm{COD}}_{\mathsf{k}}$ over 200 realizations.}
     \label{studio_1_mean}
\end{figure}

\subsection{Case study II}
The aim of this study is to analyze the effect on the performances when the noise variance is considered as an hyperparameter.
We consider again Setup 1 with \textbf{MAX-TC}, \textbf{ARX-TC}. Moreover, we introduce the additional estimators:
\begin{itemize}
    \item \textbf{MAX-S} is like the \textbf{MAX-TC} estimator but the noise variance is now an additional hyperparameter that needs to be optimized according to the approximation of the marginal likelihood;
\item \textbf{ARX-S} is like the \textbf{ARX-TC} estimator but the noise variance is an additional hyperparameter that needs to be optimized according to the marginal likelihood.
\end{itemize}

\begin{figure}
    \centering
\includegraphics[width=0.95\linewidth]{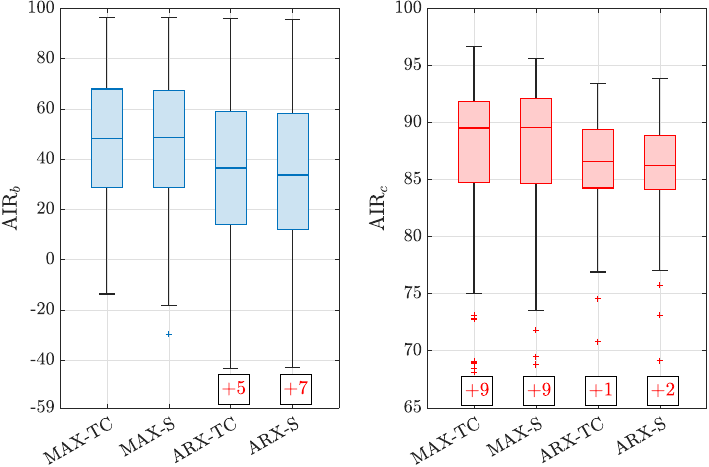}
\caption{\textit{Case study II:} \fix{ boxplots of $\rm{AIR}_b$ (left panel)  and $\rm{AIR}_c$ (right panel) over 200 simulations for the considered estimators.} The number in the box denotes the amount of outliers not shown in the picture.}
\label{studio2_air}
\end{figure}
\fix{In Figure \ref{studio2_air} the boxplots relative to the \textrm{AIR} indices for the considered estimators are depicted.} The introduction of the extra parameter for the estimation of the noise variance does not provide a significant change in the results for all the estimators.  This suggests, in particular, that the procedure to approximate the marginal likelihood in Section \ref{hyperparameter estimation} is also effective for estimating the noise variance. For sake of completeness, we performed the same study for the considered estimators, also employing the D2 kernel. The conclusions that arise from the analysis of these results are similar to the one discussed above.

\subsection{Case study III}
The aim of this study is to compare the approximation in Algorithm \ref{Alg} and the one obtained using the pseudo-linear regression model, i.e., Algorithm \ref{Alg3}. We consider \textbf{MAX-TC}, \textbf{ARX-TC} and we introduce also the following estimator:
\begin{itemize}
    \item \textbf{MAX-PL} which denotes the kernel-based estimator proposed in Section~\ref{Pseudo_linear_reg}, equipped with the TC kernel, and the approximation of the marginal likelihood is performed according to Algorithm \ref{Alg3}. 
\end{itemize}

\begin{figure}
\centering
  \includegraphics[width=0.89\linewidth]{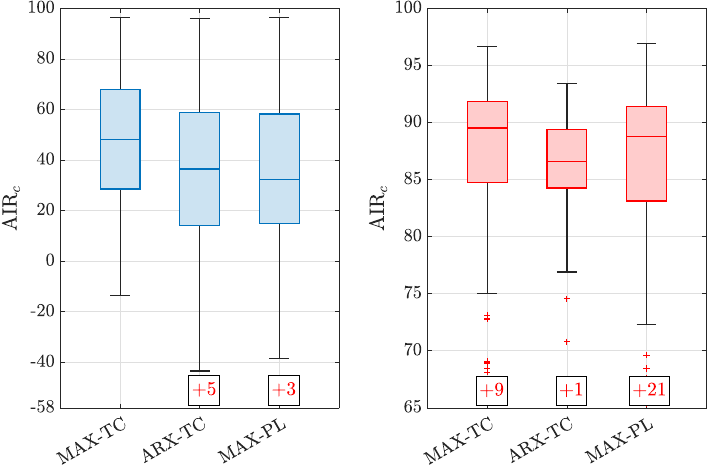}

\vspace{3mm}
  \includegraphics[width=.89\linewidth]{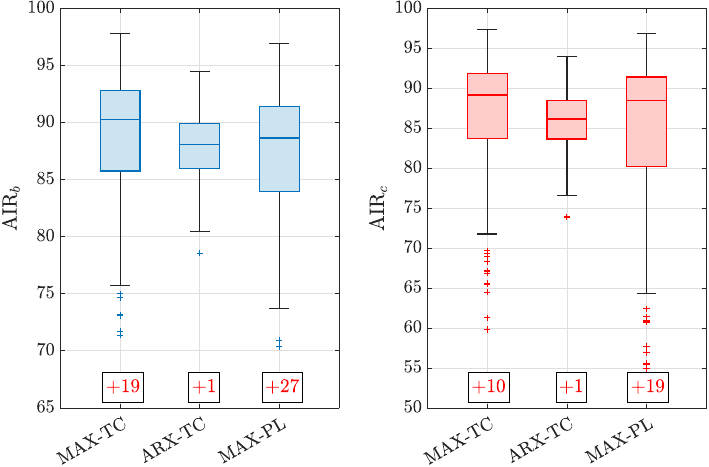}

\caption{\textit{Case-study III}:  \fix{boxplots of $\rm{AIR}_b$ (left panels)  and of $\rm{AIR}_c$ (right panels) over 200 simulations for the considered estimators, when Setup $1$ (Top) and Setup $2$ (Bottom) are used.} The numbers in the boxes denote the amount of outliers not shown in the picture.}
\label{studio3}
\end{figure}

Figure \ref{studio3} (Top) shows the boxplots relative to the \textrm{AIR} \fix{indices} for the considered estimators when Setup 1 is used. \fix{Regarding $\rm AIR_b$, the performance of  \textbf{MAX-PL} is worse than the ones of \textbf{MAX-TC} and \textbf{ARX-TC}. While, \textbf{MAX-PL} performs slightly better than \textbf{ARX-TC} in terms of $\rm AIR_c$.} Figure \ref{studio3} (Bottom) shows the $\mathrm{AIR}$ \fix{indices} using Setup 2. \fix{In this case the median of $\mathrm{AIR}_b$ and $\mathrm{AIR}_c$ indices of \textbf{\fix{MAX-PL}} are higher than the ones relative to \textbf{ARX-TC}.} However, \textbf{MAX-TC} \fix{still} performs better than \textbf{MAX-PL}.
The results of this study highlight that \textbf{MAX-PL} is definitely suboptimal in respect to \textbf{MAX-TC} in particular when the input is not so much persistently exciting, as in Setup 1. On the other hand, when the input is persistently exciting, as in Setup 2, \textbf{MAX-PL} is slightly better than \textbf{ARX-TC}. This is because the prediction error estimated in the preliminary step is more close to the actual one.

\subsection{Case study IV}
We consider Setup 3. In such scenario the impulse response $b$ is characterized by two different decay rates. Thus, the most suitable kernel for it is the integrated one \citep{PILLONETTO2016137}.
 In what follows we consider \fix{\textbf{FIR OE-TC}}, \textbf{MAX-TC}, \textbf{ARX-TC} and the additional estimator:
 \begin{itemize}
 \item \fix{{\textbf{FIR OE-INT} which denotes the Finite Impulse Response Output Error kernel-based estimator based on the parametrization (\ref{par_inf_ARX}) proposed in \cite{PILLONETTO_2011_PREDICTION_ERROR} where $a_k=0$ for any $k\in\mathbb N$ and equipped with the discrete integrated kernel \citep{ZORZI2018125}. The latter is defined as 
    \begin{align}
    \label{defInt}
    [\Kcal_{INT}]_{ij}&=\frac{\beta^{max\{i,j\}}-\alpha^{max\{i,j\}}}{2max\{i,j\}}, 
    \end{align}
where the hyperparameters $\alpha, \beta\in (0,1)$ are such that $\alpha<\beta$;}}
\item \textbf{MAX-INT} which denotes the kernel-based estimator of Section \ref{hyperparameter estimation} equipped with the integrated kernel \eqref{defInt} for $b$ and the TC kernel for $c$; the approximation of the marginal likelihood is performed according to Algorithm 1;
\item \textbf{ARX-INT} denotes the kernel-based estimator based on the parametrization (\ref{par_inf_ARX}) proposed in \cite{PILLONETTO_2011_PREDICTION_ERROR} and equipped with the integrated kernel \eqref{defInt} for $b$ and the TC kernel for $c$.
\end{itemize}

In Figure \ref{studio5} the boxplots relative to the \textrm{AIR} \fix{indices} for the considered estimators are depicted. \fix{It can be noticed that \textbf{MAX-INT} is the best estimator, followed by \textbf{MAX-TC}}. In particular, \textbf{MAX-TC} is worse than \textbf{MAX-INT} because its a priori information on $b$, given by the TC kernel, is less accurate that the one given by the integrated kernel. \fix{ \textbf{FIR OE-INT} performs better than  \textbf{ARX-INT}, suggesting that the independent parametrization in \textbf{FIR OE-INT} provides a slight performance improvement.  Finally, the performance of  \textbf{ARX-INT}} is not satisfying because the a priori information provided by the integrated kernel is associated to a wrong parametrization.
The average COD fit $\overline{ \textrm{COD}}_\mathsf{k}$ for each estimator is plotted in Figure \ref{studio5_mean}. 

\begin{figure}
\centering
\includegraphics[width=0.95\linewidth]{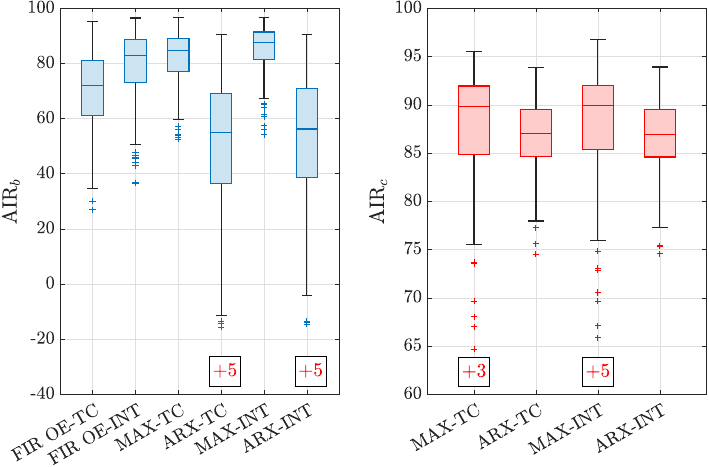}
\caption{\textit{Case-study IV:} \fix{boxplots of $\rm{AIR}_b$ (left panel)  and $\rm{AIR}_c$ (right panel) over 200 simulations for the considered estimators.} The numbers in the boxes denote the amount of outliers not shown in the picture.}
\label{studio5}
\end{figure}
\begin{figure}
    \centering
    \includegraphics[width=0.47\textwidth]{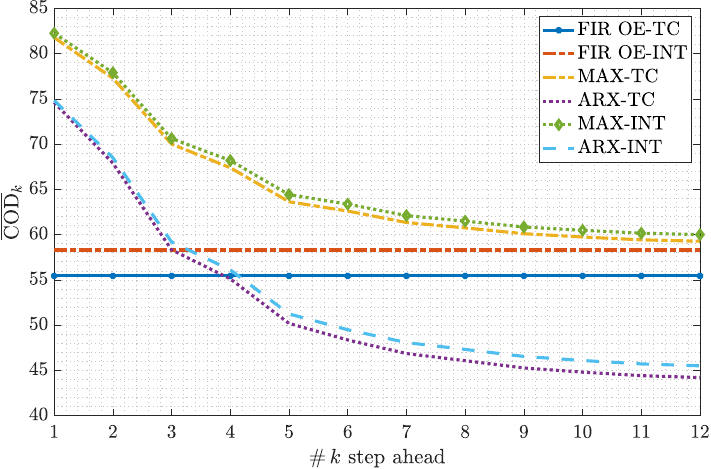}
    \caption{\textit{Case study IV}: average COD index $\overline{ \textrm{COD}}_\mathsf{k}$ over 200 realizations.}
    \label{studio5_mean}
\end{figure}
\subsection{Case study V}
We consider Setup 4 and we compare the performance of \textbf{MAX-TC} and \textbf{ARX-TC} with the one of the following estimator:
\begin{itemize}
    \item \textbf{FB-TC}: which denotes the Full-Bayesian estimator proposed in \cite{pillonetto2023full} on the parametrization (\ref{par_inf_ARX}) and equipped with the TC kernel. 
\end{itemize}
The boxplots for the $\textrm{COD}_1$ index corresponding to the considered estimators are shown in Figure \ref{comp_giapi} (top panel). We can see that \textbf{MAX-TC} performs worse than \textbf{ARX-TC} and \textbf{FB-TC}. The reason is explained below. 
The one-step ahead predictor for \textbf{ARX-TC} and \textbf{FB-TC} depends linearly on the estimated impulse responses whose decay rate is controlled by the kernel hyperparameters. This ensures that the predictor is computed in a robust way.  On the contrary,  the one-step ahead predictor for \textbf{MAX-TC} is obtained performing a filtering operation through $C(z)^{-1}$. While the decay rate of the impulse response of $C(z)$ is controlled by the kernel hyperparameters, we cannot guarantee the same for $C(z)^{-1}$. As a consequence, the computation of the one-step ahead predictor is less robust. Indeed, if we exclude the realizations corresponding to  
$\textrm{COD}_1$  outliers below -100 in at least one of the three estimators (see Figure \ref{comp_giapi}, central panel), and below -10 (see Figure \ref{comp_giapi}, bottom panel), the performance of \textbf{MAX-TC} improves significantly.

\begin{figure}
    \centering
\includegraphics[width=0.47\textwidth]
{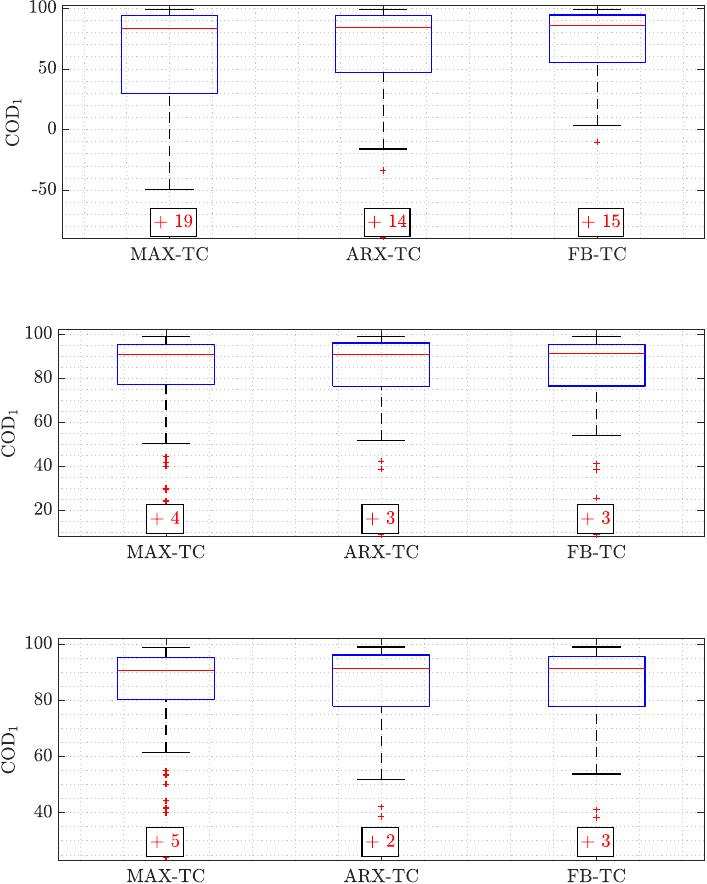}
    \caption{ \textit{Case study V}: {\it Top panel.} $\textrm{COD}_1$ index over 100 realizations. {\it Central panel.} $\textrm{COD}_1$ index after discarding outliers below -100 (82 realizations retained). {\it Bottom panel.} $\textrm{COD}_1$ index after discarding outliers below -10 (78 realizations retained). The numbers in the boxes denote the amount of outliers not shown in the picture but considered in the boxplot.}
    \label{comp_giapi}
\end{figure}

\subsection{Computational Complexity of the proposed methods}
We perform an empirical analysis of the computational cost for the evaluation of the negative log-marginal likelihood in MAX-based estimators (summarized in Algorithm \ref{Alg}, \ref{Alg3}). Moreover, we compare the latter with the one of ARX-based estimators. We  generated  a collection of datasets as in Setup 1, but here  we considered $N=1000$ samples and a square wave input signal with period of 325 samples. For each estimator we considered the TC kernel and 
for each realization, we performed a single evaluation of the marginal likelihood with hyperparameters $\lambda_b=1$, $\lambda_c=1$, $\beta_b=0.2$, $\beta_c=0.2$ and collected the total time $t_{TOT}$ needed to complete the computation.
Moreover, for MAX-based estimators, we collected two additional time periods:\\  \\
$\bullet$  $t_E$, the time needed for the  computation of $\hat{\theta}_\eta$;\\
$\bullet$ $t_\Delta=t_{TOT}-t_E$, the time needed to perform the remaining computations, beyond the one considered in $t_E$.  \\ \\
{All simulations for evaluating the computational performance were carried out on a}
DELL PC with a 13th Generation Intel\textsuperscript{\textregistered} Core\texttrademark{} i7-1355U 1.70GHz, 16 GB of RAM and Windows 11 as operating system. In Table {1}
we report the mean value corresponding to \textbf{MAX-TC}, \textbf{MAX-PL} and \textbf{ARX-TC} estimators. 
\begin{table}
    \centering
    \begin{tabular}{|c|c|c|c|}
    \hline
        &  {\bf MAX-TC} & \bf {MAX-PL} & {\bf ARX-TC}\\
        \hline
         $t_{TOT}$ [ms]& 205.5 & 4.7  & 1.8\\
         $t_E$ [ms]& 181.0 & 1.7  & -\\
         $t_\Delta$ [ms]& 24.5 & 2.9  & - \\
         \hline
    \end{tabular}
    \vspace{0.5em}
    {\small Table 1. Empirical computational costs of MAX-based and ARX-based estimators.}
\end{table}
Although $t_{TOT}$ for \textbf{MAX-TC} is 100 times the one for \textbf{ARX-TC}, most of the time of $t_{TOT}$ for \textbf{MAX-TC} is spent in computing  $\hat{\theta}_\eta$.  We recall the solution to \eqref{Theta_hat} does not admit closed form expression and require computationally expensive optimization steps.  Here, we used the Matlab routine \texttt{armax.m} to estimate $\hat{\theta}_\eta$. We believe $t_E$ can be lowered by embedding the aforementioned optimization steps directly in the routine for the evaluation of the marginal likelihood, without recurring to external functions.  The computational time  $t_{TOT}$ for \textbf{MAX-PL} and \textbf{ARX-TC} is of the same order due to the fact that the computation of  $\hat{\theta}_\eta$ in Step 2 of Algorithm \ref{Alg3} is computed in closed form. Finally, notice that $t_\Delta$ for \textbf{MAX-TC} and \textbf{MAX-PL} differs by an order of magnitude because the two algorithms requires different filtering operations  (more precisely, the filtering operations in \textbf{MAX-TC} are more complex than the ones in \textbf{MAX-PL}).

\section{Conclusions}\label{sec:conclusion}
In this paper we have presented a novel kernel-based method for the identification of nonparametric forward models. Such estimator is the solution to a nonlinear Tikhonov regularization problem for which we have proved the existence of a solution. Thus the finite dimensional approximation of the problem introduces a negligible bias in the estimate. We have considered also the problem to estimate the hyperparameters optimizing the marginal likelihood. Since the latter does not present a closed form expression, we have proposed  a procedure that relies on the Laplace approximation of the marginal likelihood.  Numerical experiments showed that the proposed estimator performs better than the classical kernel-based estimators available in the literature when the a priori information involves the forward parametrization. 

%\begin{ack}                               % Place acknowledgements
% Partially supported by the Roman Senate.  % here.
%TBA
%\end{ack}

\bibliographystyle{model5-names}        % Include this if you use bibtex 

\begin{thebibliography}{40}
\expandafter\ifx\csname natexlab\endcsname\relax\def\natexlab#1{#1}\fi
\providecommand{\bibinfo}[2]{#2}
\ifx\xfnm\relax \def\xfnm[#1]{\unskip,\space#1}\fi
%Type = Article
\bibitem[{Akaike(1974)}]{AKAIKE_1974}
\bibinfo{author}{Akaike, H.} (\bibinfo{year}{1974}).
\newblock \bibinfo{title}{{A new look at the statistical model
  identification}}.
\newblock {\it \bibinfo{journal}{IEEE Transactions on Automatic Control}\/},
  {\it \bibinfo{volume}{19}\/}, \bibinfo{pages}{716--723}.
%Type = Article
\bibitem[{Aronszajn(1950)}]{ARONSZAJN1950}
\bibinfo{author}{Aronszajn, N.} (\bibinfo{year}{1950}).
\newblock \bibinfo{title}{Theory of reproducing kernels}.
\newblock {\it \bibinfo{journal}{Trans. Amer. Math. Soc.}\/},  {\it
  \bibinfo{volume}{68}\/}, \bibinfo{pages}{337--404}.
%Type = Article
\bibitem[{Carmeli \& Toigo(2006)}]{carmeli2006vector}
\bibinfo{author}{Carmeli, E., C.and De~Vito}, \& \bibinfo{author}{Toigo, A.}
  (\bibinfo{year}{2006}).
\newblock \bibinfo{title}{Vector valued reproducing kernel {H}ilbert spaces of
  integrable functions and {M}ercer theorem}.
\newblock {\it \bibinfo{journal}{Analysis and Applications}\/},  {\it
  \bibinfo{volume}{4}\/}, \bibinfo{pages}{377--408}.
%Type = Article
\bibitem[{Chen(2018)}]{CHEN2018109}
\bibinfo{author}{Chen, T.} (\bibinfo{year}{2018}).
\newblock \bibinfo{title}{On kernel design for regularized {LTI} system
  identification}.
\newblock {\it \bibinfo{journal}{Automatica}\/},  {\it \bibinfo{volume}{90}\/},
  \bibinfo{pages}{109--122}.
%Type = Article
\bibitem[{Chen et~al.(2014)Chen, Andersen, Ljung, Chiuso \&
  Pillonetto}]{chen2014system}
\bibinfo{author}{Chen, T.}, \bibinfo{author}{Andersen, M.},
  \bibinfo{author}{Ljung, L.}, \bibinfo{author}{Chiuso, A.}, \&
  \bibinfo{author}{Pillonetto, G.} (\bibinfo{year}{2014}).
\newblock \bibinfo{title}{System identification via sparse multiple
  kernel-based regularization using sequential convex optimization techniques}.
\newblock {\it \bibinfo{journal}{IEEE Transactions on Automatic Control}\/},
  {\it \bibinfo{volume}{59}\/}, \bibinfo{pages}{2933--2945}.
%Type = Article
\bibitem[{Chen \& Ljung(2013{\natexlab{a}})}]{chen2013implementation}
\bibinfo{author}{Chen, T.}, \& \bibinfo{author}{Ljung, L.}
  (\bibinfo{year}{2013}{\natexlab{a}}).
\newblock \bibinfo{title}{Implementation of algorithms for tuning parameters in
  regularized least squares problems in system identification}.
\newblock {\it \bibinfo{journal}{Automatica}\/},  {\it \bibinfo{volume}{49}\/},
  \bibinfo{pages}{2213--2220}.
%Type = Article
\bibitem[{Chen \& Ljung(2013{\natexlab{b}})}]{CHEN20132213}
\bibinfo{author}{Chen, T.}, \& \bibinfo{author}{Ljung, L.}
  (\bibinfo{year}{2013}{\natexlab{b}}).
\newblock \bibinfo{title}{Implementation of algorithms for tuning parameters in
  regularized least squares problems in system identification}.
\newblock {\it \bibinfo{journal}{Automatica}\/},  {\it \bibinfo{volume}{49}\/},
  \bibinfo{pages}{2213--2220}.
%Type = Article
\bibitem[{Chen \& Ljung(2015{\natexlab{a}})}]{chen2015kernel1}
\bibinfo{author}{Chen, T.}, \& \bibinfo{author}{Ljung, L.}
  (\bibinfo{year}{2015}{\natexlab{a}}).
\newblock \bibinfo{title}{On kernel structures for regularized system
  identification (i): a machine learning perspective}.
\newblock {\it \bibinfo{journal}{IFAC-PapersOnLine}\/},  {\it
  \bibinfo{volume}{48}\/}, \bibinfo{pages}{1035--1040}.
%Type = Article
\bibitem[{Chen \& Ljung(2015{\natexlab{b}})}]{chen2015kernel2}
\bibinfo{author}{Chen, T.}, \& \bibinfo{author}{Ljung, L.}
  (\bibinfo{year}{2015}{\natexlab{b}}).
\newblock \bibinfo{title}{On kernel structures for regularized system
  identification (ii): A system theory perspective}.
\newblock {\it \bibinfo{journal}{IFAC-PapersOnLine}\/},  {\it
  \bibinfo{volume}{48}\/}, \bibinfo{pages}{1041--1046}.
%Type = Article
\bibitem[{Chen et~al.(2012)Chen, Ohlsson \& Ljung}]{EST_TF_REVISITED_2012}
\bibinfo{author}{Chen, T.}, \bibinfo{author}{Ohlsson, H.}, \&
  \bibinfo{author}{Ljung, L.} (\bibinfo{year}{2012}).
\newblock \bibinfo{title}{On the estimation of transfer functions,
  regularizations and gaussian processes-revisited}.
\newblock {\it \bibinfo{journal}{Automatica}\/},  {\it \bibinfo{volume}{48}\/},
  \bibinfo{pages}{1525--1535}.
%Type = Article
\bibitem[{Chen \& Pillonetto(2018)}]{chen2018stability}
\bibinfo{author}{Chen, T.}, \& \bibinfo{author}{Pillonetto, G.}
  (\bibinfo{year}{2018}).
\newblock \bibinfo{title}{On the stability of reproducing kernel hilbert spaces
  of discrete-time impulse responses}.
\newblock {\it \bibinfo{journal}{Automatica}\/},  {\it \bibinfo{volume}{95}\/},
  \bibinfo{pages}{529--533}.
%Type = Article
\bibitem[{Dalla~Libera et~al.(2021)Dalla~Libera, Carli \&
  Pillonetto}]{dalla2021kernel}
\bibinfo{author}{Dalla~Libera, A.}, \bibinfo{author}{Carli, R.}, \&
  \bibinfo{author}{Pillonetto, G.} (\bibinfo{year}{2021}).
\newblock \bibinfo{title}{Kernel-based methods for volterra series
  identification}.
\newblock {\it \bibinfo{journal}{Automatica}\/},  {\it
  \bibinfo{volume}{129}\/}, \bibinfo{pages}{109686}.
%Type = Article
\bibitem[{Darwish et~al.(2018)Darwish, Cox, Proimadis, Pillonetto \&
  T{\'o}th}]{darwish2018prediction}
\bibinfo{author}{Darwish, M.}, \bibinfo{author}{Cox, P.},
  \bibinfo{author}{Proimadis, I.}, \bibinfo{author}{Pillonetto, G.}, \&
  \bibinfo{author}{T{\'o}th, R.} (\bibinfo{year}{2018}).
\newblock \bibinfo{title}{Prediction-error identification of lpv systems: A
  nonparametric gaussian regression approach}.
\newblock {\it \bibinfo{journal}{Automatica}\/},  {\it \bibinfo{volume}{97}\/},
  \bibinfo{pages}{92--103}.
%Type = Article
\bibitem[{Fattore et~al.(2024)Fattore, Peruzzo, Sartori \& Zorzi}]{BJREG_CONF}
\bibinfo{author}{Fattore, G.}, \bibinfo{author}{Peruzzo, M.},
  \bibinfo{author}{Sartori, G.}, \& \bibinfo{author}{Zorzi, M.}
  (\bibinfo{year}{2024}).
\newblock \bibinfo{title}{A kernel-based pem estimator for forward models}.
\newblock {\it \bibinfo{journal}{IFAC-PapersOnLine}\/},  {\it
  \bibinfo{volume}{58}\/}, \bibinfo{pages}{31--36}.
\newblock \bibinfo{note}{20th IFAC Symposium on System Identification SYSID
  2024}.
%Type = Article
\bibitem[{Fujimoto(2021)}]{FUJI2}
\bibinfo{author}{Fujimoto, Y.} (\bibinfo{year}{2021}).
\newblock \bibinfo{title}{Kernel regularization in frequency domain: Encoding
  high-frequency decay property}.
\newblock {\it \bibinfo{journal}{IEEE Control Systems Letters}\/},  {\it
  \bibinfo{volume}{5}\/}, \bibinfo{pages}{367--372}.
%Type = Article
\bibitem[{Fujimoto \& Sugie(2018)}]{FUJI1}
\bibinfo{author}{Fujimoto, Y.}, \& \bibinfo{author}{Sugie, T.}
  (\bibinfo{year}{2018}).
\newblock \bibinfo{title}{Kernel-based impulse response estimation with a
  priori knowledge on the {DC} gain}.
\newblock {\it \bibinfo{journal}{IEEE Control Systems Letters}\/},  {\it
  \bibinfo{volume}{2}\/}, \bibinfo{pages}{713--718}.
%Type = Article
\bibitem[{Goodwin et~al.(1992)Goodwin, Gevers \& Ninness}]{GOODWIN_1992}
\bibinfo{author}{Goodwin, G.}, \bibinfo{author}{Gevers, M.}, \&
  \bibinfo{author}{Ninness, B.} (\bibinfo{year}{1992}).
\newblock \bibinfo{title}{Quantifying the error in estimated transfer functions
  with application to model order selection}.
\newblock {\it \bibinfo{journal}{IEEE Transactions on Automatic Control}\/},
  {\it \bibinfo{volume}{37}\/}, \bibinfo{pages}{913--928}.
%Type = Article
\bibitem[{Guo(2011)}]{guo2011inhibition}
\bibinfo{author}{Guo, D.} (\bibinfo{year}{2011}).
\newblock \bibinfo{title}{Inhibition of rhythmic spiking by colored noise in
  neural systems}.
\newblock {\it \bibinfo{journal}{Cognitive neurodynamics}\/},  {\it
  \bibinfo{volume}{5}\/}, \bibinfo{pages}{293--300}.
%Type = Article
\bibitem[{Khosravi \& Smith(2023)}]{10039069}
\bibinfo{author}{Khosravi, M.}, \& \bibinfo{author}{Smith, R.~S.}
  (\bibinfo{year}{2023}).
\newblock \bibinfo{title}{Kernel-based impulse response identification with
  side-information on steady-state gain}.
\newblock {\it \bibinfo{journal}{IEEE Transactions on Automatic Control}\/},
  {\it \bibinfo{volume}{68}\/}, \bibinfo{pages}{6401--6408}.
%Type = Article
\bibitem[{Kimeldorf \& Wahba(1970)}]{kimeldorf1970correspondence}
\bibinfo{author}{Kimeldorf, G.}, \& \bibinfo{author}{Wahba, G.}
  (\bibinfo{year}{1970}).
\newblock \bibinfo{title}{A correspondence between {B}ayesian estimation on
  stochastic processes and smoothing by splines}.
\newblock {\it \bibinfo{journal}{The Annals of Mathematical Statistics}\/},
  {\it \bibinfo{volume}{41}\/}, \bibinfo{pages}{495--502}.
%Type = Book
\bibitem[{Ljung(1999)}]{LJUNG_SYS_ID_1999}
\bibinfo{author}{Ljung, L.} (\bibinfo{year}{1999}).
\newblock {\it \bibinfo{title}{System Identification: Theory for the User}\/}.
\newblock \bibinfo{address}{New Jersey}: \bibinfo{publisher}{Prentice Hall}.
%Type = Article
\bibitem[{Ljung et~al.(2020)Ljung, Chen \&
  Mu}]{doi:10.1080/00207179.2019.1578407}
\bibinfo{author}{Ljung, L.}, \bibinfo{author}{Chen, T.}, \&
  \bibinfo{author}{Mu, B.} (\bibinfo{year}{2020}).
\newblock \bibinfo{title}{A shift in paradigm for system identification}.
\newblock {\it \bibinfo{journal}{International Journal of Control}\/},  {\it
  \bibinfo{volume}{93}\/}, \bibinfo{pages}{173--180}.
%Type = Book
\bibitem[{MacKay(2003)}]{mackay2003information}
\bibinfo{author}{MacKay, D.} (\bibinfo{year}{2003}).
\newblock {\it \bibinfo{title}{Information Theory, Inference and Learning
  Algorithms}\/}.
\newblock \bibinfo{publisher}{Cambridge University Press}.
%Type = Article
\bibitem[{Marconato et~al.(2017)Marconato, Schoukens \&
  Schoukens}]{marconato2017filter}
\bibinfo{author}{Marconato, A.}, \bibinfo{author}{Schoukens, M.}, \&
  \bibinfo{author}{Schoukens, J.} (\bibinfo{year}{2017}).
\newblock \bibinfo{title}{Filter-based regularisation for impulse response
  modelling}.
\newblock {\it \bibinfo{journal}{IET Control Theory \& Applications}\/},  {\it
  \bibinfo{volume}{11}\/}, \bibinfo{pages}{194--204}.
%Type = Book
\bibitem[{Munkres(2000)}]{munkres2000topology}
\bibinfo{author}{Munkres, J.} (\bibinfo{year}{2000}).
\newblock {\it \bibinfo{title}{Topology}\/}.
\newblock \bibinfo{publisher}{Prentice hall}.
%Type = Article
\bibitem[{Pillonetto \& Bisiacco(2024)}]{PILLONETTO2024111347}
\bibinfo{author}{Pillonetto, G.}, \& \bibinfo{author}{Bisiacco, M.}
  (\bibinfo{year}{2024}).
\newblock \bibinfo{title}{Kernel-based linear system identification: When does
  the representer theorem hold?}
\newblock {\it \bibinfo{journal}{Automatica}\/},  {\it
  \bibinfo{volume}{159}\/}, \bibinfo{pages}{111347}.
%Type = Article
\bibitem[{Pillonetto et~al.(2016)Pillonetto, Chen, Chiuso, {De Nicolao} \&
  Ljung}]{PILLONETTO2016137}
\bibinfo{author}{Pillonetto, G.}, \bibinfo{author}{Chen, T.},
  \bibinfo{author}{Chiuso, A.}, \bibinfo{author}{{De Nicolao}, G.}, \&
  \bibinfo{author}{Ljung, L.} (\bibinfo{year}{2016}).
\newblock \bibinfo{title}{Regularized linear system identification using
  atomic, nuclear and kernel-based norms: The role of the stability
  constraint}.
\newblock {\it \bibinfo{journal}{Automatica}\/},  {\it \bibinfo{volume}{69}\/},
  \bibinfo{pages}{137--149}.
%Type = Article
\bibitem[{Pillonetto \& Chiuso(2022)}]{PILLONETTO2022110169}
\bibinfo{author}{Pillonetto, G.}, \& \bibinfo{author}{Chiuso, A.}
  (\bibinfo{year}{2022}).
\newblock \bibinfo{title}{Linear system identification using the sequential
  stabilizing spline algorithm}.
\newblock {\it \bibinfo{journal}{Automatica}\/},  {\it
  \bibinfo{volume}{138}\/}, \bibinfo{pages}{110169}.
%Type = Article
\bibitem[{Pillonetto et~al.(2011)Pillonetto, Chiuso \&
  De~Nicolao}]{PILLONETTO_2011_PREDICTION_ERROR}
\bibinfo{author}{Pillonetto, G.}, \bibinfo{author}{Chiuso, A.}, \&
  \bibinfo{author}{De~Nicolao, G.} (\bibinfo{year}{2011}).
\newblock \bibinfo{title}{Prediction error identification of linear systems: A
  nonparametric gaussian regression approach}.
\newblock {\it \bibinfo{journal}{Automatica}\/},  {\it \bibinfo{volume}{47}\/},
  \bibinfo{pages}{291--305}.
%Type = Article
\bibitem[{Pillonetto \& De~Nicolao(2010)}]{PILLONETTO_DENICOLAO2010}
\bibinfo{author}{Pillonetto, G.}, \& \bibinfo{author}{De~Nicolao, G.}
  (\bibinfo{year}{2010}).
\newblock \bibinfo{title}{A new kernel-based approach for linear system
  identification}.
\newblock {\it \bibinfo{journal}{Automatica}\/},  {\it \bibinfo{volume}{46}\/},
  \bibinfo{pages}{81--93}.
%Type = Article
\bibitem[{Pillonetto \& Ljung(2023)}]{pillonetto2023full}
\bibinfo{author}{Pillonetto, G.}, \& \bibinfo{author}{Ljung, L.}
  (\bibinfo{year}{2023}).
\newblock \bibinfo{title}{Full bayesian identification of linear dynamic
  systems using stable kernels}.
\newblock {\it \bibinfo{journal}{Proceedings of the National Academy of
  Sciences}\/},  {\it \bibinfo{volume}{120}\/}, \bibinfo{pages}{e2218197120}.
%Type = Book
\bibitem[{Rasmussen \& Williams(2006)}]{RASMUSSEN_WILLIAMNS_2006}
\bibinfo{author}{Rasmussen, C.}, \& \bibinfo{author}{Williams, C.}
  (\bibinfo{year}{2006}).
\newblock {\it \bibinfo{title}{{Gaussian Processes for Machine Learning}}\/}.
\newblock \bibinfo{publisher}{The MIT Press}.
%Type = Article
\bibitem[{Romeres et~al.(2019)Romeres, Zorzi, Camoriano, Traversaro \&
  Chiuso}]{romeres2019derivative}
\bibinfo{author}{Romeres, D.}, \bibinfo{author}{Zorzi, M.},
  \bibinfo{author}{Camoriano, R.}, \bibinfo{author}{Traversaro, S.}, \&
  \bibinfo{author}{Chiuso, A.} (\bibinfo{year}{2019}).
\newblock \bibinfo{title}{Derivative-free online learning of inverse dynamics
  models}.
\newblock {\it \bibinfo{journal}{IEEE Transactions on Control Systems
  Technology}\/},  {\it \bibinfo{volume}{28}\/}, \bibinfo{pages}{816--830}.
%Type = Article
\bibitem[{Schwarz(1978)}]{SCHWARZ_1978}
\bibinfo{author}{Schwarz, G.} (\bibinfo{year}{1978}).
\newblock \bibinfo{title}{{Estimating the Dimension of a Model}}.
\newblock {\it \bibinfo{journal}{The Annals of Statistics}\/},  {\it
  \bibinfo{volume}{6}\/}, \bibinfo{pages}{461--464}.
%Type = Article
\bibitem[{da~Silva \& Vilela(2015)}]{SilvaViella}
\bibinfo{author}{da~Silva, L.~A.}, \& \bibinfo{author}{Vilela, R.~D.}
  (\bibinfo{year}{2015}).
\newblock \bibinfo{title}{Colored noise and memory effects on formal spiking
  neuron models}.
\newblock {\it \bibinfo{journal}{Phys. Rev. E}\/},  {\it
  \bibinfo{volume}{91}\/}, \bibinfo{pages}{062702}.
%Type = Book
\bibitem[{S\"{o}derstr\"{o}m \& Stoica(1989)}]{SODERSTROM_STOICA_1988}
\bibinfo{author}{S\"{o}derstr\"{o}m, T.}, \& \bibinfo{author}{Stoica, P.}
  (\bibinfo{year}{1989}).
\newblock {\it \bibinfo{title}{System Identification}\/}.
\newblock \bibinfo{address}{Hemel Hempstead, UK}:
  \bibinfo{publisher}{Prentice-Hall International}.
%Type = Book
\bibitem[{Zeidler(1995)}]{zeidler1995applied}
\bibinfo{author}{Zeidler, E.} (\bibinfo{year}{1995}).
\newblock {\it \bibinfo{title}{Applied Functional Analysis: Main Principles and
  Their Applications}\/} volume \bibinfo{volume}{109}.
\newblock \bibinfo{publisher}{Springer Science \& Business Media}.
%Type = Article
\bibitem[{Zorzi(2022)}]{BKRON}
\bibinfo{author}{Zorzi, M.} (\bibinfo{year}{2022}).
\newblock \bibinfo{title}{Nonparametric identification of {K}ronecker
  networks}.
\newblock {\it \bibinfo{journal}{Automatica}\/},  {\it
  \bibinfo{volume}{145}\/}, \bibinfo{pages}{110518}.
%Type = Article
\bibitem[{Zorzi(2024)}]{zorzi2021second}
\bibinfo{author}{Zorzi, M.} (\bibinfo{year}{2024}).
\newblock \bibinfo{title}{A second-order generalization of {TC} and {DC}
  kernels}.
\newblock {\it \bibinfo{journal}{IEEE Transactions on Automatic Control}\/},
  {\it \bibinfo{volume}{69}\/}, \bibinfo{pages}{3835--3848}.
%Type = Article
\bibitem[{Zorzi \& Chiuso(2018)}]{ZORZI2018125}
\bibinfo{author}{Zorzi, M.}, \& \bibinfo{author}{Chiuso, A.}
  (\bibinfo{year}{2018}).
\newblock \bibinfo{title}{The harmonic analysis of kernel functions}.
\newblock {\it \bibinfo{journal}{Automatica}\/},  {\it \bibinfo{volume}{94}\/},
  \bibinfo{pages}{125--137}.

\end{thebibliography}

\end{document}